\documentclass[11pt,twoside]{article}
\usepackage[margin=1.3in]{geometry}

\usepackage[T1]{fontenc}      

\usepackage{amsfonts,amsmath,amssymb,amsthm}
\usepackage{cases,color,curves}

\usepackage{cite}

\usepackage{bm,enumerate,graphicx,ifthen,latexsym,
            listings,makeidx,mathrsfs,psfrag,times,xcolor}

\usepackage[colorlinks,citecolor=green,linkcolor=red]{hyperref}
\usepackage[title]{appendix}

\def\MR#1{}   

\makeindex

\newtheorem{theorem}{Theorem}[section]
\newtheorem{lemma}[theorem]{Lemma}

\newtheorem{remark}{Remark}[section]

\begin{document}

\title{\bf On the structure of isolated singularities for semilinear elliptic equations}

\author{%
  Meiqing Xu  \and  Hui Yang
}

\date{\today}
\maketitle

\begin{abstract}
In this paper, we study isolated singularities of the following semilinear elliptic equation
\[
 -\Delta u+\frac12 x\cdot \nabla u+\frac{1}{q-1}u-u^q=0 ~~~~~~ \textmd{in } \Omega \setminus \{0\},
\]
where $n\ge 3$, $\Omega \subset \mathbb{R}^n$ is a domain, $0 \in \Omega$ and $q>1$. This equation arises in the study of blow-up profiles of semilinear heat equations. For $\frac{n}{n-2}< q < \frac{n+2}{n-2}$, we establish a complete classification of isolated singularities for nonnegative solutions and characterize the precise asymptotic behavior of singular solutions. Our results improve those of Guedda and  Kirane (Trans. Amer. Math. Soc., 1995: 3595-3603), where analogous results were obtained only for radially symmetric positive solutions. In addition, we also derive the asymptotic behavior of solutions in the Serrin critical case $q=\frac{n}{n-2}$ and the supercritical case $q>\frac{n+2}{n-2}$.

\medskip


\medskip

\end{abstract}


\section{Introduction}
In this paper, we study the asymptotic behavior near isolated singularities of nonnegative solutions to the following elliptic equation
\begin{equation}\label{eq:elliptic}
    -\Delta u+\frac12 x\cdot \nabla u+\frac{1}{q-1}u-u^q=0 ~~~~~~ \textmd{in } \Omega \setminus \{0\},
\end{equation}
where $\Omega \subset \mathbb{R}^n$ is a domain, $0 \in \Omega$ and $q>1$. Equation \eqref{eq:elliptic} is the stationary elliptic equation associated with
the similarity variables for the semilinear heat equation
\begin{equation}\label{eq:heat}
    v_t=\Delta v+|v|^{q-1}v.
\end{equation}
Indeed, near a possible blow-up point $(a, T)$, set
\[
    y:=\frac{x-a}{\sqrt{T-t}},\qquad  s:=-\ln(T-t),
\]
and
\[
    w(y, s):=(T-t)^{\frac{1}{q-1}} v(x, t).
\]
Then $w$ satisfies the semilinear parabolic equation
\begin{equation*}
        w_s=\Delta w-\frac12 y\cdot\nabla w
    -\frac{1}{q-1}w+|w|^{q-1}w.
\end{equation*}
Thus, self-similar blow-up profiles are governed by the stationary solutions of this rescaled equation. In the case of positive solutions, it coincides with equation \eqref{eq:elliptic}.

A fundamental problem for semilinear heat equations is to understand the asymptotic behavior of solutions near the singularity, particularly the determination of blow-up profiles. The classical work of Giga and Kohn \cite{MR784476} showed that the type I blow-up of
\eqref{eq:heat} is asymptotically self-similar in the subcritical and critical ranges. A key ingredient of their approach is the classification of bounded entire solutions of equation \eqref{eq:elliptic}. More precisely, they proved that for $n \leq 2$ or for $q \le\frac{n+2}{n-2}$, the only bounded global solutions of the equation
\begin{equation}\label{eq:elliptic002}
-\Delta w+\frac12 x\cdot\nabla w+\frac{1}{q-1}w-|w|^{q-1}w=0  \qquad \text{in } \mathbb R^n
\end{equation}
are $w \equiv 0$ and $w \equiv \pm (\frac{1}{q-1})^{\frac{1}{q-1}}$. In subsequent works \cite{MR876989,giga1989nondegeneracy}, Giga and Kohn investigated the nondegeneracy of these blow-up profiles in the subcritical regime. Filippas and Kohn \cite{FilippasKohnCPAM} employed a center manifold analysis to derive refined asymptotics for these profiles and to relate them to the local geometry of the blow-up set. The stability of the blow-up profile was studied by Merle and Zaag \cite{MerleZaagDuke}. In the supercritical case, the situation becomes considerably more delicate, since the steady-state equation \eqref{eq:elliptic002} admits positive radially symmetric solutions \cite{BuddQiJDE,TroySIAM,Lepin,BiernatBizo,CRSmem}. This leads to the emergence of new type I blow-up phenomena (see, e.g., \cite{CRSmem,MatanoMerleJFA}). Furthermore, for $q \geq q_{JL}$ (the Joseph-Lundgren exponent), new type II blow-up solutions also appear (see, e.g., \cite{HerreroVelzquez94,CollotAPDE,MizoguchiMathAnn,SekiJFA18}). A recent breakthrough by Collot, Merle and Rapha\"{e}l \cite{CollotJAMS20} established the existence of strongly anisotropic type II blow-up at an isolated point in this context. For more results on blow-up of semilinear heat equations, we refer to \cite{MR2060042,CdeMJEMS20,deMWZPisa,MerleZaagJEMS26,MatanoMerleCPAM,VelzquezTAMS93,ZaagDuke06} and the references therein.

The structure of type II blow-up solutions of \eqref{eq:heat} in \cite{CollotAPDE,MizoguchiMathAnn,SekiJFA18} is deeply related to the singular solution of \eqref{eq:elliptic002} given by
\[
\Phi(x) := \lambda_{n,q} |x|^{-\frac{2}{q-1}}, ~~~~~~\lambda_{n,q}:=\left[ \frac{2}{q-1} \left( n-2-\frac{2}{q-1} \right) \right]^{\frac{1}{q-1}}.
\]
Consequently, the singular solutions of the steady-state equation \eqref{eq:elliptic002} are of significant importance for understanding the blow-up phenomena of semilinear heat equations. This paper is devoted to the study of isolated singularities of \eqref{eq:elliptic}, which also lays the foundation for the classification of singular solutions of \eqref{eq:elliptic002}. For $1< q < \frac{n}{n-2}$, Guedda and Kirane \cite{gueddaKirane1995} classified the isolated singularities of positive solutions to \eqref{eq:elliptic}, thereby extending earlier results of Brezis-Lions \cite{BrezisLions81} and Lions \cite{lions1980isolated} for the Lane-Emden equation. When $\frac{n}{n-2} < q < \frac{n+2}{n-2}$, they proved that if $u$ is a positive {\it radial} solution of \eqref{eq:elliptic} (assuming that $\Omega=B_1$), then
\begin{enumerate}
    \item [(1)] $\lim\limits_{r\to0}r^{\frac{2}{q-1}}u(r)=\ell$ and $\ell \in\{0, \lambda_{n,q}\}$, where $\lambda_{n,q} :=\left[ \frac{2}{q-1} \left(n-2-\frac{2}{q-1}\right) \right]^{\frac{1}{q-1}}$.

    \item [(2)] when $\ell=0$ and $q>2$, the singularity is removable, i.e., $u$ can be extended to $\Omega$ as a $C^2$ solution of \eqref{eq:elliptic} in $\Omega$.
    \end{enumerate}
Our first result removes both the radial symmetry assumption and the condition $q>2$ in the above result of Guedda and Kirane \cite{gueddaKirane1995}.
Without loss of generality, we take $\Omega=B_1$ to be the unit ball centered at the origin, and consider
\begin{equation}\label{eq:main}
    -\Delta u+\frac12 x\cdot \nabla u+\frac{1}{q-1}u-u^q=0
    \qquad \text{in } B_1\setminus\{0\},
\end{equation}
where $B_1 \setminus\{0\} \subset \mathbb{R}^n$ is the punctured unit ball, $q >1$ and $n\geq 3$. Then we have the following theorem.

\begin{theorem}\label{thm:isolated-singularity}
Let $\frac{n}{n-2}<q<\frac{n+2}{n-2}$ and $u\in C^2(B_1\setminus\{0\})$ be a nonnegative solution of \eqref{eq:main}.
Then either $u$ can be extended as a $C^2(B_1)$ solution of \eqref{eq:main} in $B_1$, or $u$ is a distributional solution in $B_1$ and satisfies
\begin{equation*}
    \lim_{|x|\to 0}|x|^{{\frac{2}{q-1}}}u(x)= \left[ \frac{2}{q-1} \left(n-2-\frac{2}{q-1}\right) \right]^{\frac{1}{q-1}}.
\end{equation*}
\end{theorem}

We emphasize that Theorem \ref{thm:isolated-singularity} does not assume radial symmetry of the solution or impose any boundary conditions. The main difficulty here is that the drift term breaks the symmetry and scaling invariance of the equation, so the classical methods for the Lane-Emden equation (see Caffarelli-Gidas-Spruck \cite{CGS} and Gidas-Spruck\cite{MR615628})
\begin{equation}\label{eq:lane-emden}
    -\Delta u=u^q
    \qquad \text{in } B_1\setminus\{0\}.
\end{equation}
cannot be applied directly. To address these difficulties, we will employ local blow-up analysis, establish a monotonicity formula, and construct new barrier functions to complete the proof of Theorem \ref{thm:isolated-singularity}.

We now turn to the Serrin critical case $q=\frac{n}{n-2}$ and prove the following result.

\begin{theorem}\label{thm:main2}
Let \(q=\frac{n}{n-2}\) and \(u\in C^2(B_1\setminus\{0\})\) be a nonnegative solution of
\eqref{eq:main}. Then either $u$ can be extended as a $C^2(B_1)$ solution of \eqref{eq:main} in $B_1$, or $u$ satisfies
\begin{equation*}
    \lim_{x\to0}
    |x|^{n-2}(-\ln |x|)^{\frac{n-2}{2}}u(x)
    =
    \left(\frac{n-2}{\sqrt2}\right)^{n-2}.
\end{equation*}
\end{theorem}

For the Lane-Emden equation \eqref{eq:lane-emden} with $q=n/(n-2)$, Aviles \cite{aviles1987local,Aviles1983} has established an analogous classification result. In our setting, however, the drift term \(\frac12 x\cdot\nabla u\) and the linear term \(\frac{n-2}{2}u\) destroy the exact homogeneity and produce extra lower-order perturbations after passing to logarithmic cylindrical coordinates. Hence, the corresponding energy identities, angular-decay estimates, and removable-singularity argument of Aviles\cite{aviles1987local,Aviles1983} require substantial modification. Specifically, we need to establish a logarithmic upper bound adapted to the drift operator, control the decay of angular oscillation, prove the integrability of both \(v_t\) and \(\nabla_\theta v\) in the presence of
the extra \(e^{-2t}\) and \(t^{-2}\) error terms, and construct a new logarithmic barrier function for the zero-limit case.

Finally, we consider the supercritical case $q>\frac{n+2}{n-2}$, which was studied by Bidaut-V\'eron and V\'eron
\cite{bidautVeronVeron1991} for the Lane-Emden equation \eqref{eq:lane-emden}. Under an appropriate upper bound assumption, we establish the following asymptotic result.

\begin{theorem}\label{thm:supercritical-alternative}
Let $q>\frac{n+2}{n-2}$ and $u\in C^2(B_1\setminus\{0\})$ be a nonnegative solution of \eqref{eq:main}. Assume that
\begin{equation}\label{eq:supercritical-scale-bound}
    |x|^{\frac{2}{q-1}} u(x)\in L^\infty_{\mathrm{loc}}(B_1).
\end{equation}
Then either $u$ can be extended as a $C^2(B_1)$ solution of \eqref{eq:main} in $B_1$, or there exists a positive function
$\omega\in C^\infty(\mathbb S^{n-1})$ satisfying
\begin{equation}\label{eq:supercritical-spherical-profile}
    \Delta_{\mathbb S^{n-1}} \omega
    -
    \frac{2}{q-1}
    \left(
        n-2-\frac{2}{q-1}
    \right)\omega
    +
    \omega^q=0
    \qquad\text{on }\mathbb S^{n-1}
\end{equation}
such that for every integer $k\ge0$,
\[
    r^{\frac{2}{q-1}} u(r,\cdot) \to \omega
    \quad\text{in }C^k(\mathbb S^{n-1}) \text{ as  } r \to 0.
\]
\end{theorem}

\begin{remark}
For positive solutions of equation \eqref{eq:supercritical-spherical-profile}, Bidaut-V\'eron and V\'eron \cite{bidautVeronVeron1991} proved that only constant solutions exist when $\frac{n+2}{n-2}<q<\frac{n+1}{n-3}$, while Dancer, Guo and Wei \cite{DancerGW12} constructed infinitely many nonconstant solutions when $\frac{n+1}{n-3}< q < p_c(n-1)$, with $p_c(n-1) $ being the Joseph-Lundgren exponent in dimension $n-1$. When $q=\frac{n+1}{n-3}$, equation \eqref{eq:supercritical-spherical-profile} reduces to the Yamabe equation on $\mathbb S^{n-1}$, whose positive solutions were completely classified by Obata \cite{Obata}.
\end{remark}

The proof of Theorem \ref{thm:supercritical-alternative} relies on a perturbed energy identity and a careful application of Simon's convergence theorem to a weighted variational functional with exponentially decaying error terms. To handle removable singularities, we avoid the semigroup representation of the linearized equation in Bidaut-V\'eron and V\'eron \cite{bidautVeronVeron1991}, and instead construct elementary exponential barriers on the cylinder. This approach provides a more robust proof of the removable singularity in our setting.

In the Sobolev critical case $q=\frac{n+2}{n-2}$, the drift term in equation \eqref{eq:main} makes it difficult to apply the moving plane method, and consequently,  the study of isolated singularities becomes more challenging. We will investigate this case in our subsequent work \cite{XYsubsequent}. For the critical Lane-Emden equation, also known as the Yamabe equation, its isolated singularities have been extensively studied; see \cite{CGS,chenc,Li,MR1666838,MarquesCVPDE,HanLiLiCPAM,XiongZhangIMRN,HanXiongZhangJFA} and the references therein. We also refer to \cite{AoJRAM20,CJSX2014,chenLin1999asymptotic,DdGWMathAnn,MR3869387,JX21Poincare,taliaferroZhang2006asymptotic,li2006fullyNonlinear,hanLiTeixeira2010asymptotic} for isolated singularities of other Yamabe-type equations, and to \cite{AndradeOJDE,AndradedoONonlinearity,CdoS2019,ChenLinMathAnn,GKSAPDE,YangZouZAMP} for isolated singularities of semilinear elliptic systems.

This paper is organized as follows. In Section \ref{sec pre}, we establish a criterion for the removability of singularities and a monotonicity formula, which will serve as essential tools for proving Theorems \ref{thm:isolated-singularity}, \ref{thm:main2} and \ref{thm:supercritical-alternative}. Sections \ref{sec:subcritical},  \ref{sec:borderline} and \ref{sec:supercritical} are devoted to the proofs of Theorems \ref{thm:isolated-singularity}, \ref{thm:main2} and \ref{thm:supercritical-alternative}, respectively. Throughout the following proofs, we assume that the nonnegative solution $u$ is nontrivial. The strong maximum principle then yields \(u>0\) in \(B_1\setminus\{0\}\).

\section{Removable singularities and monotonicity formula}\label{sec pre}
In this section, we establish the removability of singularities and a monotonicity formula, which are essential tools in the proofs of Theorems \ref{thm:isolated-singularity}, \ref{thm:main2} and \ref{thm:supercritical-alternative}. First, we prove that any nonnegative solution of \eqref{eq:main} with $q\geq \frac{n}{n-2}$ is a distributional solution on the whole $B_1$. While this was already known from Guedda and Kirane \cite{gueddaKirane1995}, our proof is more direct, following an idea from Yang \cite{YangCVPDE20}.

\begin{lemma}\label{lem:local-integrability-distribution}
Let $q\geq \frac{n}{n-2}$ and let $u\in C^2(B_1\setminus\{0\})$ be a nonnegative solution of
\eqref{eq:main}. Then $u\in L^q_{\mathrm{loc}}(B_1)$.
Moreover, $u$ satisfies \eqref{eq:main} in the sense of distributions in
$B_1$, i.e., for every $\varphi\in C_c^\infty(B_1)$,
\begin{equation*}
\int_{B_1}u\left(
        -\Delta\varphi
        -\frac12 x\cdot\nabla\varphi
        +\left(\frac1{q-1}-\frac n2\right)\varphi
    \right)\,dx
    -\int_{B_1}u^q\varphi\,dx=0.
\end{equation*}
\end{lemma}

\begin{proof}
We first prove the local $L^q$ estimate. Fix $0<r<1$ and choose
$0<\varepsilon<r/4$.  Let $\eta\in C^\infty[0,\infty)$ satisfy
$0\leq\eta\leq1$, $\eta=0$ on $[0,1]$, $\eta=1$ on $[2,\infty)$, and
set
\[
    \eta_\varepsilon(x):=\eta\left(\frac{|x|}{\varepsilon}\right),
    \qquad
    \xi_\varepsilon:=\eta_\varepsilon^{\frac{2q}{q-1}}.
\]
Then $\xi_\varepsilon=0$ in $B_\varepsilon$, $\xi_\varepsilon=1$ in
$\mathbb R^n\setminus B_{2\varepsilon}$ and
\begin{equation}\label{eq:xi-derivative-bounds}
    |\nabla\xi_\varepsilon|
    \leq C\varepsilon^{-1}\xi_\varepsilon^{1/q}\chi_{{B_{2\varepsilon}\setminus B_\varepsilon}},
    \qquad
    |\Delta\xi_\varepsilon|
    \leq C\varepsilon^{-2}\xi_\varepsilon^{1/q}\chi_{{B_{2\varepsilon}\setminus B_\varepsilon}}.
\end{equation}
Multiplying \eqref{eq:main} by $\xi_\varepsilon$ and integrating over
$B_r$ gives
\begin{equation}\label{eq:local-Lq-identity}
\begin{aligned}
I_\varepsilon(r)
&:=\int_{B_r}u^q\xi_\varepsilon\,dx                                      \\
&=\int_{B_r}u(-\Delta\xi_\varepsilon)\,dx
  -\frac12\int_{B_r}u\,x\cdot\nabla\xi_\varepsilon\,dx                    \\
&\quad
  +\left(\frac1{q-1}-\frac n2\right)\int_{B_r}u\xi_\varepsilon\,dx
  +\int_{\partial B_r}\left(\frac r2u-\partial_\nu u\right)\,dS.
\end{aligned}
\end{equation}
The boundary term is finite and independent of $\varepsilon$. Hence it
is bounded above by a constant $C_r$. Also, the third term on the right-hand side of
\eqref{eq:local-Lq-identity} is nonpositive.
Hence
\begin{equation*}
    I_\varepsilon(r)
    \leq C_r
    +\left|\int_{B_r}u\Delta\xi_\varepsilon\,dx\right|
    +\frac12\left|\int_{B_r}u\,x\cdot\nabla\xi_\varepsilon\,dx\right|.
\end{equation*}
Using \eqref{eq:xi-derivative-bounds} and H\"older's inequality, we get
\[
\begin{aligned}
\left|\int_{B_r}u\Delta\xi_\varepsilon\,dx\right|
&\leq C\varepsilon^{-2}
    \int_{B_{2\varepsilon}\setminus B_\varepsilon}u\xi_\varepsilon^{1/q}\,dx                         \\
&\leq C\varepsilon^{-2}|B_{2\varepsilon}\setminus B_\varepsilon|^{1-1/q}
    \left(\int_{B_{2\varepsilon}\setminus B_\varepsilon}u^q\xi_\varepsilon\,dx\right)^{1/q}           \\
&\leq C\varepsilon^{n(1-1/q)-2} I_\varepsilon(r)^{1/q}
\leq C I_\varepsilon(r)^{1/q}
\end{aligned}
\]
because of $n(1-1/q)-2\geq0$. Similarly,
\[
\begin{aligned}
\left|\int_{B_r}u\,x\cdot\nabla\xi_\varepsilon\,dx\right|
&\leq C\int_{B_{2\varepsilon}\setminus B_\varepsilon}u\xi_\varepsilon^{1/q}\,dx
 \leq C I_\varepsilon(r)^{1/q}.
\end{aligned}
\]
Thus
\[
    I_\varepsilon(r)\leq C_r+C I_\varepsilon(r)^{1/q}.
\]
By Young's inequality, we obtain
\begin{equation*}
    I_\varepsilon(r)\leq C_r,
\end{equation*}
where the constant is independent of $\varepsilon$.  Letting
$\varepsilon \to 0$ and using Fatou's lemma, we have
\[
    \int_{B_r}u^q\,dx<\infty.
\]
Since $0<r<1$ is arbitrary, $u\in L^q_{\mathrm{loc}}(B_1)$.

It remains to show that $u$ satisfies \eqref{eq:main} in the sense of distributions in
$B_1$.  For any $\varphi\in C_c^\infty(B_1)$, using the same
$\eta_\varepsilon$ as above, the function
$\eta_\varepsilon\varphi$ is compactly supported in
$B_1\setminus\{0\}$. Therefore
\[
\int_{B_1}(\eta_\varepsilon\varphi)
\left(-\Delta u+\frac12 x\cdot\nabla u+\frac1{q-1}u-u^q\right)\,dx=0.
\]
Integrating by parts, the above equation becomes
\begin{equation}\label{eq:cutoff-distribution-identity}
\begin{aligned}
0
&=\int_{B_1}u\eta_\varepsilon\left(
        -\Delta\varphi
        -\frac12 x\cdot\nabla\varphi
        +\left(\frac1{q-1}-\frac n2\right)\varphi
    \right)\,dx
    -\int_{B_1}u^q\eta_\varepsilon\varphi\,dx                               \\
&\quad
    -\int_{B_1}u\left(
        \varphi\Delta\eta_\varepsilon
        +2\nabla\eta_\varepsilon\cdot\nabla\varphi
        +\frac12\varphi\,x\cdot\nabla\eta_\varepsilon
    \right)\,dx.
\end{aligned}
\end{equation}
Applying H\"older's inequality, we have
\[
\begin{aligned}
\left|\int_{B_1}u\varphi\Delta\eta_\varepsilon\,dx\right|
&\leq C\varepsilon^{-2}\int_{B_{2\varepsilon}\setminus B_\varepsilon}u\,dx                         \\
&\leq C\varepsilon^{n(1-1/q)-2}
    \left(\int_{B_{2\varepsilon}\setminus B_\varepsilon}u^q\,dx\right)^{1/q}
\to 0,
\end{aligned}
\]
where we have used $n(1-1/q)-2\geq0$ and
$u^q\in L^1(B_r)$.  Similarly, we also have
\[
\left|\int_{B_1}u\nabla\eta_\varepsilon\cdot\nabla\varphi\,dx\right|
\leq C\varepsilon^{n(1-1/q)-1}
    \left(\int_{B_{2\varepsilon}\setminus B_\varepsilon}u^q\,dx\right)^{1/q}
\to 0
\]
and
\[
\left|\int_{B_1}u\varphi\,x\cdot\nabla\eta_\varepsilon\,dx\right|
\leq C\varepsilon^{n(1-1/q)}
    \left(\int_{B_{2\varepsilon}\setminus B_\varepsilon}u^q\,dx\right)^{1/q}
\to 0.
\]
By letting $\varepsilon \to 0$ in \eqref{eq:cutoff-distribution-identity}, we arrive at the desired equality.
\end{proof}

We now establish a removability theorem for singularities. Its proof relies on the construction of barrier functions and an iterative application of the maximum principle.

\begin{lemma}\label{lem:zero-removable}
Let $q>\frac{n}{n-2}$ and let
$u\in C^2(B_1\setminus\{0\})$ be a nonnegative solution of \eqref{eq:main}. If
\begin{equation}\label{eq:zero-limit}
    \lim_{|x|\to 0}|x|^{\frac{2}{q-1}}u(x)=0,
\end{equation}
then $u$ can be extended  as a $C^2(B_1)$ solution of \eqref{eq:main} in $B_1$.
\end{lemma}
\begin{proof}

We write equation \eqref{eq:main} as
\begin{equation*}
        \mathcal{L} u=0,
        \qquad
        \mathcal{L}:=-\Delta +\frac12 x\cdot\nabla+\frac{1}{q-1}-u^{q-1}.
\end{equation*}
For $s>0$, a direct computation yields
\begin{equation*}
\begin{aligned}
            \mathcal{L}(|x|^{-s})
       & =\left(s(n-2-s)|x|^{-2}+\frac{1}{q-1}-\frac{s}{2}-u^{q-1}(x)\right)|x|^{-s} \\
        &=\left(s(n-2-s)|x|^{-2}+\frac{1}{q-1}-\frac{s}{2}-o(|x|^{-2})\right)|x|^{-s}   \qquad \text{as } x \to 0.
\end{aligned}
\end{equation*}
Since $q>\frac{n}{n-2}$, we have for $0<r<1$ small enough that
\begin{equation}\label{eq:L-power}
        \mathcal{L}(|x|^{-\frac{2}{q-1}})>0,
        \qquad
        \mathcal{L}(|x|^{-\frac{1}{q-1}})>0
        \qquad \text{in } B_r\setminus\{0\}.
\end{equation}
For $0<\varepsilon<r$, define
\[
        A:=r^\frac{1}{q-1} \max_{\partial B_r}u,\qquad m_\varepsilon:=\varepsilon^\frac{2}{q-1} \max_{\partial B_\varepsilon}u.
\]
By \eqref{eq:zero-limit}, we have $m_\varepsilon\to 0$ as $\varepsilon \to 0$. Consider the barrier
\[
        \Phi_\varepsilon(x):=A|x|^{-\frac{1}{q-1}}+m_\varepsilon |x|^{-\frac{2}{q-1}}.
\]
Then
\begin{equation*}
    \mathcal{L}\Phi_\varepsilon>0=\mathcal{L} u\qquad\text{in }A_{\varepsilon,r}:=\{x:\varepsilon<|x|<r\}.
\end{equation*}
Moreover,
\[
        \Phi_\varepsilon\ge u \qquad \text{on } \partial A_{\varepsilon,r}.
\]
We claim that
\begin{equation*}
     \Phi_\varepsilon \ge u\qquad \text{for } x\in A_{\varepsilon,r}.
\end{equation*}
Suppose otherwise. Then the quotient
\[
        z:=\frac{u}{\Phi_\varepsilon}
\]
has an interior maximum point $x_0\in A_{\varepsilon,r}$ such that
\[
        z(x_0)=\max_{\overline A_{\varepsilon,r}}\frac{u}{\Phi_\varepsilon}>1.
\]
At this point,
\[
        \nabla z(x_0)=0
        \qquad \mbox{and}  \qquad
        \Delta z(x_0)\le 0.
\]
Writing $u=z\Phi_\varepsilon$, we obtain
\begin{align*}
        \mathcal{L} u
        &=\mathcal{L}(z\Phi_\varepsilon)  \\
        &=-2\nabla z\cdot\nabla\Phi_\varepsilon
          +z(-\Delta\Phi_\varepsilon)
          +\Phi_\varepsilon(-\Delta z) \\
        &\quad
          +\frac12\Phi_\varepsilon x\cdot\nabla z
          +\frac12z x\cdot\nabla\Phi_\varepsilon
          +\left(\frac{1}{q-1}-u^{q-1}\right)z\Phi_\varepsilon.
\end{align*}
Evaluating at $x_0$ gives
\[
        0=(\mathcal{L} u)(x_0)
        =z(x_0)(\mathcal{L}\Phi_\varepsilon)(x_0)+\Phi_\varepsilon(x_0)(-\Delta z)(x_0)>0,
\]
because $z(x_0)>0$, $\Phi_\varepsilon(x_0)>0$, $\mathcal{L}\Phi_\varepsilon(x_0)>0$ and $-\Delta z(x_0)\ge0$.
This contradiction proves
\begin{equation*}
        u(x)\le A|x|^{-\frac{1}{q-1}}+m_\varepsilon |x|^{-\frac{2}{q-1}}
        \qquad \text{for } \varepsilon<|x|<r.
\end{equation*}
Letting $\varepsilon \to 0$ gives
\begin{equation}\label{eq:first-improvement}
        u(x)\le A |x|^{-\frac{1}{q-1}}
        \qquad \text{for } 0<|x|<r.
\end{equation}

We now use the maximum principle again to derive that $u$ is bounded near the origin.
By \eqref{eq:zero-limit}, we know
\[
    A^{q-1}r
    =
    \left(r^{\frac{2}{q-1}}\max_{\partial B_r}u\right)^{q-1}
    \to0
    \qquad\text{as }r \to 0.
\]
Thus, taking \(r>0\) sufficiently small, we may assume that \(A^{q-1}r\) is arbitrarily small.
Using \eqref{eq:first-improvement}, for $0<|x|<r$ we have
\begin{equation*}
    u^{q-1}(x)(r-|x|)\le \frac{A^{q-1}(r-|x|)}{|x|}\le \frac{A^{q-1}r}{|x|}.
\end{equation*}
Then
\begin{align*}
        \mathcal{L}(r-|x|)
        &=\frac{n-1}{|x|}-\frac{|x|}{2}+\frac{1}{q-1}(r-|x|)-u^{q-1}(x)(r-|x|)  \\
    &\ge
    \frac{n-1-A^{q-1}r}{|x|}-\frac{|x|}{2}      \ge
    \frac{n-1-A^{q-1}r-\frac12r^2}{|x|}
    \ge \frac{c_0}{|x|},
\end{align*}
where $c_0>0$, provided $r>0$ is chosen sufficiently small. Let $M:=\max\limits_{\partial B_r}u$.
Since
\[
        \mathcal{L} M=(\frac{1}{q-1}-u^{q-1})M\ge -\frac{A^{q-1}M}{|x|},
\]
we can choose $K>0$ so large that
\begin{equation}\label{eq:bounded-super-part}
        \mathcal{L}\left(M+K(r-|x|)\right)\ge 0
        \qquad \text{in }B_r\setminus\{0\}.
\end{equation}
For $0<\varepsilon<r$, define
\[
        \Psi_\varepsilon(x):=M+K(r-|x|)+A \varepsilon^\frac{1}{q-1} |x|^{-\frac{2}{q-1}}.
\]
By \eqref{eq:L-power} and \eqref{eq:bounded-super-part}, we have
\[
        \mathcal{L}\Psi_\varepsilon>0
        \qquad \text{in }A_{\varepsilon,r}.
\]
Notice that
\[
\Psi_\varepsilon\ge M\ge u
\qquad \text{on } \partial B_r
\]
 and
\begin{equation*}
    \Psi_\varepsilon\ge A \varepsilon^{-\frac{1}{q-1}} \ge u \qquad \text{on $\partial B_\varepsilon$.}
\end{equation*}
By a maximum principle argument similar to the one above, we obtain
\[
        u(x)\le M+K(r-|x|)+A \varepsilon^\frac{1}{q-1} |x|^{-\frac{2}{q-1}}
        \qquad \text{for }\varepsilon<|x|<r.
\]
Sending $\varepsilon \to  0$, we have
\begin{equation*}
        u(x)\le  M+Kr
        \qquad \text{for }0<|x|<r.
\end{equation*}
Thus, $u$ is bounded near the origin.

Finally, we apply elliptic estimates to show that the solution $u$ is $C^2$ near the origin. Let \(x_0\in B_{1/2}\setminus\{0\}\) and set
\[
        \rho:=\frac{|x_0|}{4},\qquad
        U(y):=u(x_0+\rho y),\quad |y|<1.
\]
Then \(B_\rho(x_0)\subset B_1 \setminus\{0\}\) and
\(U\) satisfies
\[
    -\Delta_y U
    +\frac{\rho}{2}(x_0+\rho y)\cdot\nabla_y U
    +\frac{\rho^2}{q-1}U
    =
    \rho^2 U^q
    \qquad\text{in }B_1.
\]
The standard interior gradient estimate gives
$
        |\nabla_y U(0)|\le C.
$
Returning to the original variables, we obtain
$
        |x_0|\,|\nabla u(x_0)|
        \le C.
$
Thus
\begin{equation}\label{eq:xgrad-bound-subcritical}
        |x|\,|\nabla u(x)|\le C
        \qquad\text{near the origin}.
\end{equation}
Set
\[
        F(x)
        :=
        \frac12 x\cdot\nabla u(x)
        +\frac{1}{q-1}u(x)
        -u(x)^q,
        \qquad x\ne0.
\]
Then \(F\in L_{\mathrm{loc}}^\infty (B_1)\).  By Lemma  \ref{lem:local-integrability-distribution}, the interior $W^{2,p}$ estimate and the Schauder estimate, we deduce that $u \in C^{2}_{\rm loc}(B_1)$. This completes the proof of Lemma \ref{lem:zero-removable}.
\end{proof}

Next, we establish a monotonicity formula for the classification of isolated singularities. This is inspired by \cite{GKSAPDE,YangCVPDE20}. Let
$u\in C^2(B_1\setminus\{0\})$ be a nonnegative solution of
\eqref{eq:main}. We introduce the cylindrical variables
\begin{equation}\label{eq:cylindrical-coordinates}
t = \ln r, \quad  r=|x|, \quad \theta = \frac{x}{|x|} \in \mathbb S^{n-1},
\end{equation}
and set
\begin{equation}\label{eq:cylindrical-rescaling}
    w(t,\theta):=r^\frac{2}{q-1} u(r\theta).
    \qquad
\end{equation}
Then $w$ satisfies
\begin{equation*}
    -w_{tt}
    +\left(k_1+\frac12 e^{2t}\right)w_t
    +k_0w
    -w^q
    -\Delta_{\mathbb S^{n-1}} w
    =0
    \qquad
    \text{in }(-\infty,0)\times\mathbb S^{n-1},
\end{equation*}
where
\begin{equation}\label{eq:cylindrical-coefficient}
      k_0:=\frac{2}{q-1}(n-2-\frac{2}{q-1})
    \quad \mbox{and} \quad
    k_1:=\frac{4}{q-1}-(n-2).
\end{equation}
We define the energy $E(t; w)$ as
\begin{equation}\label{eq:cylindrical-energy-def}
    E(t; w):=
    \int_{\mathbb S^{n-1}}
    \left(
        -\frac12w_t^2
        +\frac{k_0}{2}w^2
        -\frac{1}{q+1}w^{q+1}
        +\frac12|\nabla_\theta w|^2
    \right)
    \,d\theta.
\end{equation}
Then the following result shows that $E(t; w)$ is monotone and has a limit as $t \to -\infty$.

\begin{lemma}\label{lem:cylindrical-energy-limit}
Let \(q>1\) and \(k_0,k_1\in\mathbb R\). Let
\(w\in C^2((-\infty,0)\times\mathbb S^{n-1})\) be a nonnegative solution of
\begin{equation}\label{eq:general-cylinder-equation}
    -w_{tt}
    +\left(k_1+\frac12e^{2t}\right)w_t
    +k_0w
    -w^q
    -\Delta_{\mathbb S^{n-1}} w
    =0
    \qquad\text{in }(-\infty,0)\times\mathbb S^{n-1}.
\end{equation}
Suppose that there exist \(T_0<0\) and \(C_0>0\) such that
\begin{equation}\label{eq:bdd}
    0\le w(t,\theta)\le C_0
    \qquad\text{on }(-\infty,T_0]\times\mathbb S^{n-1}.
\end{equation}
Then $E(t; w)$ is monotone for $t$ sufficiently negative, and the limit
\[
   \bar E:=\lim_{t\to-\infty}E(t;w)
\]
exists.
\end{lemma}

\begin{proof}
 By the standard elliptic estimates, we have
\begin{equation}\label{eq:cylindrical-energy-boundedness-assumption}
    \sup_{(t,\theta)\in(-\infty,T_1]\times\mathbb S^{n-1}}
    \left(
        |w(t,\theta)|
        +|w_t(t,\theta)|
        +|\nabla_\theta w(t,\theta)|
    \right)<\infty
\end{equation}
for some $T_1 < T_0$. Differentiating \eqref{eq:cylindrical-energy-def} and integrating by
parts on \(\mathbb S^{n-1}\), we obtain
\[
\begin{aligned}
    \frac{d}{dt}E(t;w)
    &=
    \int_{\mathbb S^{n-1}}
    \left(
        -w_{tt}
        +k_0w
        -w^q
        -\Delta_{\mathbb S^{n-1}} w
    \right)w_t
    \,d\theta  \\
    &=
    -\left(k_1+\frac12e^{2t}\right)
    \int_{\mathbb S^{n-1}}w_t^2\,d\theta,
\end{aligned}
\]
where the last equality follows from
\eqref{eq:general-cylinder-equation}.  If \(k_1\ge0\), then
\(E(t;w)\) is nonincreasing.  If \(k_1<0\), we choose \(T\le T_1\) so
negative that
\[
    k_1+\frac12e^{2t}\le\frac{k_1}{2}<0
    \qquad\text{for }t\le T.
\]
Then \(E(t;w)\) is nondecreasing on \((-\infty,T]\).
Finally, \eqref{eq:cylindrical-energy-boundedness-assumption} implies
that \(E(t;w)\) is bounded for \(t\le T_1\).  Hence, its limit as
\(t\to-\infty\) exists and is finite.
\end{proof}

The next lemma shows that every sequence of time translations of \(w\) admits a subsequence converging to a limit independent of \(t\).

\begin{lemma}\label{lem:constant-energy-stationarity}
Under the assumptions of Lemma \ref{lem:cylindrical-energy-limit}, suppose further that \(k_1\ne0\).
Then, for every sequence \(s_j\to-\infty\), after passing to a
subsequence, there exists a nonnegative function
\(W\in C^2(\mathbb S^{n-1})\) satisfying
\begin{equation}\label{eq:limit}
    -\Delta_{\mathbb S^{n-1}} W+k_0W-W^q=0
    \qquad\text{on }\mathbb S^{n-1}
\end{equation}
such that
\[
    w(t+s_j,\theta) \to W(\theta)
    \quad\text{in }C^2_{\mathrm{loc}}
    (\mathbb R\times\mathbb S^{n-1}).
\]
Moreover,
\[
E(t;W) =\lim_{j\to\infty}E(t+s_j;w) =\bar E.
\]
\end{lemma}

\begin{proof}
Let
\[
    w_j(t,\theta):=w(t+s_j,\theta).
\]
Then, by the Schauder estimates, after passing to a subsequence, we have
\[
    w_j \to  W
    \quad\text{in }C^2_{\mathrm{loc}}
    (\mathbb R\times\mathbb S^{n-1})
\]
for some nonnegative function \(W \in C^2_{\mathrm{loc}}
    (\mathbb R\times\mathbb S^{n-1})\).  Since
\[
    -(w_j)_{tt}
    +\left(k_1+\frac12e^{2(t+s_j)}\right)(w_j)_t
    +k_0w_j
    -w_j^q
    -\Delta_{\mathbb S^{n-1}} w_j
    =0,
\]
and \(e^{2(t+s_j)}\to0\) locally uniformly,  passing to the limit gives
\begin{equation}\label{eq:limit2}
      -W_{tt}+k_1W_t+k_0W-W^q-\Delta_{\mathbb S^{n-1}} W=0
    \qquad\text{in }\mathbb R\times\mathbb S^{n-1}.
\end{equation}
For every fixed \(t\in\mathbb R\), Lemma~\ref{lem:cylindrical-energy-limit} yields
\[
    E(t;W)
    =\lim_{j\to\infty}E(t;w_j)
    =\lim_{j\to\infty}E(t+s_j;w)
    =\bar E.
\]
Thus, while $W$ may depend on the chosen subsequence \(s_j\to-\infty\), its energy $E(t; W)$ is independent of this choice and is determined solely by the original solution \(w\). Hence, \(E(t;W)=\bar E\) is constant with respect to $t$. From the computation in the proof of
Lemma~\ref{lem:cylindrical-energy-limit}, we obtain
\[
    0 \equiv \frac{d}{dt}E(t;W)
    =-k_1\int_{\mathbb S^{n-1}}W_t^2\,d\theta.
\]
Since \(k_1\ne0\), it follows that \(W_t\equiv0\).  Consequently, $W$ is independent of $t$, and we may write \(W=W(\theta)\). Lemma \ref{lem:constant-energy-stationarity} is proved.
\end{proof}

\section{Proof of Theorem~\ref{thm:isolated-singularity}}\label{sec:subcritical}
In this section, we prove Theorem \ref{thm:isolated-singularity}. First, we establish the following upper bound estimate by using local blow-up analysis and a Liouville theorem.
\begin{lemma}\label{lem:first-apriori}
Let $1<q<\frac{n+2}{n-2}$ and $u\in C^2(B_1\setminus\{0\})$ be a nonnegative solution of \eqref{eq:main}. Then there exists a constant $C=C(n,q)>0$ such that
\begin{equation}\label{eq:first-apriori}
    u(x)\leq C |x|^{-\frac{2}{q-1}}
    \qquad \text{for  } 0<|x|\leq \frac12.
\end{equation}
\end{lemma}
\begin{proof}
By the maximum principle, we may assume that $u>0$ in $B_1\setminus\{0\}$. Suppose, to the contrary, that \eqref{eq:first-apriori} fails. Then there exist nonnegative solutions
$u_k$ of \eqref{eq:main} and points
$x_k\in B_{1/2}\setminus\{0\}$ such that
\begin{equation}\label{eq:mkxk-large}
    M_k(x_k)|x_k|>2k,
\end{equation}
where $M_k(x):=u_k(x)^{\frac{q-1}{2}}$.
We apply the doubling lemma in Pol\'a\v{c}ik-Quittner-Souplet \cite[Lemma 5.1]{MR2350853} in the complete
metric space $X:=\overline B_1$, with
\[
    \Gamma:=\{0\}\cup\partial B_1
    \quad \mbox{and}  \quad
    D:=X\setminus\Gamma=B_1\setminus\{0\}.
\]
Since $x_k\in B_{1/2}\setminus\{0\}$, we have
\[
    d(x_k,\Gamma)=|x_k|.
\]
Therefore \eqref{eq:mkxk-large} gives
\[
    M_k(x_k)d(x_k,\Gamma)>2k.
\]
By the doubling lemma, there exist points
$y_k\in B_1\setminus\{0\}$ such that
\begin{equation}\label{eq:doubling-points}
    M_k(y_k)d(y_k,\Gamma)>2k,
    \quad
    M_k(y_k)\geq M_k(x_k),
\end{equation}
and
\begin{equation*}
    M_k(z)\leq 2M_k(y_k)
    \quad
    \text{for all }
    z\in B_{k/M_k(y_k)}(y_k).
\end{equation*}
Since $\frac{k}{M_k(y_k)}<\frac12 d(y_k,\Gamma)$,
the ball $B_{k/M_k(y_k)}(y_k)$ is compactly contained in
$B_1\setminus\{0\}$. Hence, we obtain
\begin{equation*}
    u_k(z)\leq 2^{\frac{2}{q-1}}u_k(y_k)
    \quad
    \text{for all }
    z\in B_{k/M_k(y_k)}(y_k).
\end{equation*}
Let
\[
    \lambda_k:=M_k(y_k)^{-1}.
\]
By \eqref{eq:doubling-points}, we have $ \lambda_k\to 0$. Define the rescaled functions
\begin{equation*}
    v_k(x):=\lambda_k^{\frac{2}{q-1}}
    u_k(y_k+\lambda_k x),
    \quad x\in B_k.
\end{equation*}
Then $v_k(0)=1$ and $0\leq v_k\leq 2^{\frac{2}{q-1}}$ in $B_k$. Moreover, $v_k$ satisfies
\begin{equation*}
    -\Delta v_k
    +\frac12\lambda_k(y_k+\lambda_k x)\cdot\nabla v_k
    +\frac{\lambda_k^2}{q-1}v_k
    -v_k^q=0
    \quad \text{in }B_k.
\end{equation*}
By the standard interior elliptic estimates, after passing to a subsequence,
\[
    v_k\to v
    \quad \text{in } C^2_{\mathrm{loc}}(\mathbb R^n),
\]
where $v \in C^2_{\mathrm{loc}}(\mathbb R^n)$ is a nonnegative entire solution of
\begin{equation}\label{eq:liouville-limit}
    -\Delta v=v^q
    \quad \text{in }\mathbb R^n.
\end{equation}
Moreover, $v(0)=1$. This contradicts the Liouville theorem of Gidas-Spruck \cite{MR615628}. Thus, the estimate \eqref{eq:first-apriori} is proved.
\end{proof}

We now apply the cylindrical change of variables introduced in \eqref{eq:cylindrical-coordinates} and \eqref{eq:cylindrical-rescaling}. In the current case $\frac{n}{n-2} < q < \frac{n+2}{n-2}$, the coefficients given by \eqref{eq:cylindrical-coefficient} satisfy \(k_0>0\) and
\(k_1>0\).

\begin{lemma}\label{lem:asymptotic-alternative}
Let $\frac{n}{n-2}<q<\frac{n+2}{n-2}$ and $u\in C^2(B_1\setminus\{0\})$ be a nonnegative solution of
\eqref{eq:main}.
Then the following alternative holds:
\begin{equation*}
\lim_{|x|\to0}|x|^{\frac{2}{q-1}}u(x)
    =
    0
    ~~ \text{or} ~~\lambda_{n,q},
\end{equation*}
where
 \[
 \lambda_{n,q} :=\left[ \frac{2}{q-1} \left(n-2-\frac{2}{q-1}\right) \right]^{\frac{1}{q-1}}.
 \]
\end{lemma}

\begin{proof}
Lemma~\ref{lem:first-apriori} gives
\[
    0\le w(t,\theta)\le C(n,q)
    \qquad\text{for }t\le-\ln2.
\]
By Lemma~\ref{lem:cylindrical-energy-limit}, the limit
$
    \bar E:=\lim\limits_{t\to-\infty}E(t;w)
$
exists and is finite. Let \(s_j\to-\infty\) be arbitrary.  It follows from Lemma \ref{lem:constant-energy-stationarity} that, after passing
to a subsequence, there exists a nonnegative function
\(W\in C^2(\mathbb S^{n-1})\) such that
\[
    w(t+s_j,\theta) \to W(\theta)
    \quad \text{in }C^2_{\mathrm{loc}}
    (\mathbb R\times\mathbb S^{n-1}).
\]
Moreover, \(E(t;W)\equiv \bar E\).  Define
\[
    U(x):=|x|^{-\frac{2}{q-1}}
    W\!\left(\frac{x}{|x|}\right),
    \qquad x\in\mathbb R^n\setminus\{0\}.
\]
By \eqref{eq:limit}, $U$ satisfies
\[
    -\Delta U=U^q
    \qquad\text{in }\mathbb R^n\setminus\{0\}.
\]
By Caffarelli-Gidas-Spruck \cite[Theorem 8.1 and Corollary 8.2]{CGS}, either \(U\equiv0\) or
\(U\) is radially symmetric with respect to the origin. It follows that $W$ is a nonnegative constant, denoted by $\ell$. From \eqref{eq:limit}, we then obtain $\ell\in\{0,\lambda_{n,q}\}$.

It remains to prove that \(w(t,\theta)\) converges uniformly to \(\ell\) as \(t\to-\infty\).  Suppose otherwise.  Then there exist
\(\varepsilon_0>0\), a sequence \(s_j\to-\infty\), and points
\(\theta_j\in\mathbb S^{n-1}\) such that
\[
    |w(s_j,\theta_j)-\ell|\ge\varepsilon_0.
\]
Applying Lemma \ref{lem:constant-energy-stationarity} to this sequence
and repeating the argument above, after passing to a subsequence we
obtain some \(\widetilde\ell\in\{0,\lambda_{n,q}\}\) with $\widetilde\ell \not= \ell$ such that
\[
    w(t+s_j,\theta)\to\widetilde\ell
    \quad\text{in }C^2_{\mathrm{loc}}
    (\mathbb R\times\mathbb S^{n-1}).
\]
Since \(\bar E\) depends only on $w$ and is independent of the chosen sequence \(s_j\to-\infty\), we have
\[
    E(t;\ell)=E(t;\widetilde\ell)=\bar E.
\]
However,
\[
    E(t;0)\equiv 0
    \quad \mbox{and}  \quad
    E(t;\lambda_{n,q})
   \equiv
    \frac{(q-1)|\mathbb S^{n-1}|\lambda_{n,q}^{q+1}}{2(q+1)}
    >0.
\]
Hence, we have \(\widetilde\ell=\ell\). This is a contradiction. Consequently,
\[
\lim_{|x|\to0}|x|^{\frac{2}{q-1}}u(x)
=\ell\in\{0,\lambda_{n,q}\}.
\]
This finishes the proof of Lemma \ref{lem:asymptotic-alternative}.
\end{proof}

\begin{proof}[Proof of Theorem \ref{thm:isolated-singularity}]
This is an immediate consequence of Lemmas \ref{lem:asymptotic-alternative} and \ref{lem:zero-removable}.
\end{proof}

\section{Proof of Theorem~\ref{thm:main2}}\label{sec:borderline}
In this section, we prove Theorem \ref{thm:main2} with the assumption that
$
    q=\frac{n}{n-2}.
$
Define the spherical average of a function $f$ by
\[
    \bar f(r)
    :=
    \frac{1}{|\mathbb{S}^{n-1}|}
    \int_{\mathbb{S}^{n-1}}f(r\theta)\,d\theta.
\]
First, we establish the following upper bound estimate, which improves the result in Lemma  \ref{lem:first-apriori} when $q=\frac{n}{n-2}$.

\begin{lemma}\label{lem:spherical-average}
Let $q=\frac{n}{n-2}$ and \(u\in C^2(B_1\setminus\{0\})\) be a nonnegative solution of
\eqref{eq:main}. Then there exist constants \(C=C(n)>0\) and \(T_0=T_0(u)>0\)
such that
\begin{equation}\label{eq:pointwise-log-upper}
    u(x)\le C |x|^{2-n}(-\ln |x|)^{-\frac{n-2}{2}},
    \qquad 0<|x|<e^{-T_0}.
\end{equation}
\end{lemma}

\begin{proof}
The spherical average equation and Jensen's inequality give
\begin{equation*}
    \bar u''
    +\left(\frac{n-1}{r}-\frac r2\right)\bar u'
    -\frac{n-2}{2}\bar u
    +\bar u^{\frac{n}{n-2}}
    \le0.
\end{equation*}
Let
\[
   w(t):=|x|^{n-2}\bar u(|x|),
    \quad
    t = -\ln |x|.
\]
A direct calculation yields
\begin{equation}\label{eq:w-ineq}
    w_{tt}
    +
    \left((n-2)+\frac12e^{-2t}\right)w_t
    +
    w^{\frac{n}{n-2}}
    \le0.
\end{equation}
Set
$
    a(t):=(n-2)+\frac12e^{-2t}.
$
Let
\[
    A(t):=\int_{T}^{t}a(s)\,d s
    \quad  \text{and}  \quad
    Y(t):=e^{A(t)}w_t(t),
\]
where \(T>0\) is fixed.  From \eqref{eq:w-ineq}, we have
\begin{equation}\label{eq:Yprime}
    Y'(t)
    =
    e^{A(t)}\bigl(w_{tt}+a(t)w_t\bigr)
    \le
    -e^{A(t)}w(t)^{\frac{n}{n-2}}
    <0.
\end{equation}
We claim that \(w_t (t)<0\) for all sufficiently large \(t\).  Indeed, if
\(w_t(t_1)<0\) for some \(t_1\), then \(Y(t_1)<0\). Since \(Y\) is
decreasing, it follows that
\[
    w_t(t)=e^{-A(t)}Y(t)<0,\quad t\ge t_1.
\]
Thus, we only need to prove that \(w_t<0\) at some large point. Suppose otherwise that
\(w_t(t)\ge0\) for all large \(t\). Then
\(w(t)\ge c>0\) for all large \(t\).  From \eqref{eq:Yprime}, we have
\begin{equation}\label{eq:39}
  Y(t)\le Y(T)-c^{\frac{n}{n-2}}\int_T^t e^{A(s)}\,d s.
\end{equation}
Since \(A(t)=(n-2)t+O(1)\), the right-hand side becomes negative for large
\(t\), and hence \(w_t(t)<0\), a contradiction. Hence, we have \(w_t(t)<0\) for all large \(t\).  It follows that
$
    \ell:=\lim_{t\to\infty}w(t)
$
exists and \(\ell\ge0\).  We prove that \(\ell=0\).  Assume to the contrary that
\(\ell>0\).  Then \(w(t)\ge c>0\) for all large \(t\). Following the same argument as in \eqref{eq:39}, we obtain
\[
    w_t(t)
    =
    e^{-A(t)}Y(t)
    \le
    e^{-A(t)}Y(T)
    -
    c^{\frac{n}{n-2}} e^{-A(t)}
    \int_T^t e^{A(s)}\,d s.
\]
Since
\[
    e^{-A(t)}\int_T^t e^{A(s)}\,d s \to \frac{1}{n-2},
\]
we get \(w_t(t)\le -c_0<0\) for all sufficiently large \(t\).  This would
force \(w(t)\) to become negative, which is impossible. Hence \(\ell=0\).

Let \(\delta>0\). Since \(a(t)\to n-2\), there exists \(T_\delta>0\) such
that $a(t)\le n-2+\delta $ for $t\ge T_\delta$. Hence, we have
\begin{equation*}
    -w_{tt}-(n-2+\delta)w_t\ge w^{\frac{n}{n-2}}, ~~~ t\ge T_\delta.
\end{equation*}
Multiplying by \(e^{(n-2+\delta)t}\) and integrating from \(T_\delta\) to
\(t\), we get
\[
    -w_t(t)
    \ge
    -e^{-(n-2+\delta)(t-T_\delta)}w_t(T_\delta)
    +
    \int_{T_\delta}^{t}
        e^{-(n-2+\delta)(t-s)}w(s)^{\frac{n}{n-2}}\,d s.
\]
Since \(w\) is decreasing, \(w(s)\ge w(t)\) for \(s\le t\).  Hence
\[
    -w_t(t)
    \ge
    w(t)^{\frac{n}{n-2}}
    \int_{T_\delta}^{t}
        e^{-(n-2+\delta)(t-s)}\,d s,
\]
from which we deduce that
\begin{equation}\label{eq:wt-liminf}
    \liminf_{t\to\infty}
    \frac{-w_t(t)}{w(t)^{\frac{n}{n-2}}}
    \ge
    \frac1{n-2+\delta}.
\end{equation}
Thus, for every \(\eta\in(0,1)\) and every
\(\delta>0\), there exists \(T_{\eta,\delta}\) such that
\[
    \frac{d}{d t}w(t)^{-\frac{2}{n-2}}  =
    \frac{2}{n-2}(-w_t)w^{-\frac{n}{n-2}}
    \ge \frac{2(1-\eta)}{(n-2)(n-2+\delta)},
    \quad t\ge T_{\eta,\delta}.
\]
Integrating from \(T_{\eta,\delta}\) to \(t\) gives
\[
    w(t)^{-\frac{2}{n-2}}
    \ge
    C_{\eta,\delta}
    +
   \frac{2(1-\eta)}{(n-2)(n-2+\delta)}t.
\]
Consequently,
\[
    \limsup_{t\to\infty} t^{\frac{n-2}{2}}w(t)
    \le
    \left(
        \frac{(n-2)(n-2+\delta)}{2(1-\eta)}
    \right)^{\frac{n-2}{2}}.
\]
Letting first \(\eta \to 0\) and then \(\delta \to 0\),
we conclude that
\begin{equation}\label{eq:lemma1-limsup}
    \limsup_{r \to 0}
    (-\ln r)^{\frac{n-2}{2}}r^{n-2}\bar u(r)
    \le
    \left(\frac{n-2}{\sqrt2}\right)^{n-2}.
\end{equation}

For \(0<r<r_0<\frac12\), define
\[
    V^{(r)}(y):=r^{n-2}u(ry),
    \quad y\in B_2\setminus \overline{B_{1/2}}.
\]
By Lemma~\ref{lem:first-apriori}, for \(y\in B_2\setminus \overline{B_{1/2}}\), we have $
0\le V^{(r)}(y)\le C_0 |y|^{2-n}\le C_1$, where \(C_1\) is independent of \(r\).  Moreover, \(V^{(r)}\) satisfies
\[
    -\Delta V^{(r)}+b_r(y)\cdot\nabla V^{(r)}+c_r(y)V^{(r)}=0
    \quad\text{in }B_2\setminus \overline{B_{1/2}}
\]
with
\[
    b_r(y):=\frac{r^2}{2}y
    \quad \text{and}  \quad
    c_r(y):=\frac{(n-2)r^2}{2}-V^{(r)}(y)^{\frac{2}{n-2}}.
\]
The preceding bound gives $\|b_r\|_{L^\infty(B_2\setminus \overline{B_{1/2}})}+\|c_r\|_{L^\infty(B_2\setminus \overline{B_{1/2}})}\le C_2$,
where \(C_2\) is independent of \(r\). Hence, by the Harnack inequality, there exists a constant \(C\), independent of \(r\), such that
\[
    \sup_{B_{3/2}\setminus \overline{B_{2/3}}}V^{(r)}\le C \inf_{B_{3/2}\setminus \overline{B_{2/3}}}V^{(r)}.
\]
This implies that
\[
    \sup_{|x|=r}u(x)
    \le C \inf_{|x|=r}u(r),
    \quad 0<r<r_0.
\]
This, together with \eqref{eq:lemma1-limsup}, leads to the desired conclusion.
\end{proof}

For \(t=-\ln r\), define
\begin{equation*}
    \phi(t,\theta):=r^{n-2} u(r\theta)
    \quad \text{and} \quad
    v(t,\theta):=t^{(n-2)/2}\phi(t,\theta),
    \quad \theta\in\mathbb{S}^{n-1}.
\end{equation*}
The next lemma shows that \(v(t,\cdot)\) approaches its spherical average
in \(L^2(\mathbb{S}^{n-1})\) at the rate \(t^{-1}\).

\begin{lemma}\label{lem:spherical-oscillation}
There exists a constant
\(C>0\) such that
\begin{equation*}
    \|v(t,\cdot)-\bar v(t)\|_{L^2(\mathbb{S}^{n-1})}
    \le C t^{-1} \quad   \text{for } t\ge T_0.
\end{equation*}
\end{lemma}

\begin{proof}
We first derive the equation satisfied by \(\phi\). A direct computation gives
\begin{equation}\label{eq:phi-equation}
    \phi_{tt}
    +
    \left((n-2)+\frac12e^{-2t}\right)\phi_t
    +
    \Delta_{\mathbb S^{n-1}}\phi
    +
    \phi^{\frac{n}{n-2}}
    =
    0.
\end{equation}
By Lemma~\ref{lem:spherical-average}, we have
\begin{equation}\label{eq:phi-pointwise}
    0\le \phi(t,\theta)\le C t^{-(n-2)/2},
    \quad t\ge T_0, ~~ \theta\in\mathbb{S}^{n-1}.
\end{equation}
Let
\[
    \psi(t,\theta):=\phi(t,\theta)-\bar\phi(t)  \quad \text{and} \quad  X(t):=\int_{\mathbb{S}^{n-1}}\psi(t,\theta)^2\,d\theta.
\]
Then $\int_{\mathbb{S}^{n-1}}\psi(t,\theta)\,d\theta=0$.
Taking the spherical average of \eqref{eq:phi-equation} and subtracting the
resulting equation from \eqref{eq:phi-equation}, we obtain
\begin{equation}\label{eq:psi-equation}
    \psi_{tt}
    +((n-2)+\frac12e^{-2t})\psi_t
    +\Delta_{\mathbb S^{n-1}}\psi
    +
    \phi^{\frac{n}{n-2}}-\overline{\phi^{\frac{n}{n-2}}}
    =
    0.
\end{equation}
Multiplying \eqref{eq:psi-equation} by \(2\psi\) and integrating over
\(\mathbb{S}^{n-1}\), we get
\begin{equation}\label{eq:X-identity}
    X''(t)+((n-2)+\frac12e^{-2t})X'(t)
    =
    2\int_{\mathbb{S}^{n-1}}\psi_t^2\,d\theta
    +
    2\int_{\mathbb{S}^{n-1}}|\nabla_\theta\psi|^2\,d\theta
    -
    2\int_{\mathbb{S}^{n-1}}\phi^{\frac{n}{n-2}}\psi\,d\theta.
\end{equation}
By the Poincar\'e inequality on the sphere,
\begin{equation}\label{eq:poincare-sphere}
    \int_{\mathbb{S}^{n-1}}|\nabla_\theta\psi|^2\,d\theta
    \ge
    (n-1)\int_{\mathbb{S}^{n-1}}\psi^2\,d\theta
    =
    (n-1)X(t).
\end{equation}
Moreover, by \eqref{eq:phi-pointwise}, together with Cauchy's inequality and Young's inequality, we have
\begin{equation}\label{eq:phi-2q}
    2\left|\int_{\mathbb{S}^{n-1}}\phi^{\frac{n}{n-2}}\psi\,d\theta\right|
    \le
    (n-1)X(t)
    +
    C\int_{\mathbb{S}^{n-1}}\phi^{\frac{2n}{n-2}}\,d\theta
    \le
    (n-1)X(t)
    +
    Ct^{-n}.
\end{equation}
Combining \eqref{eq:X-identity}, \eqref{eq:poincare-sphere} and
\eqref{eq:phi-2q}, we obtain
\begin{equation}\label{eq:X-diff-ineq}
    X''(t)
    +
    \Big((n-2)+\frac12e^{-2t}\Big)X'(t)
    -
    (n-1)X(t)
    \ge
    -C_1 t^{-n},
    \quad t\ge T_0.
\end{equation}
We also observe that \eqref{eq:phi-pointwise} implies
\begin{equation}\label{eq:X-to-zero}
    0\le X(t)\le C t^{-(n-2)}\to0
    \quad     \text{as }t\to\infty.
\end{equation}
We now compare \(X(t)\) with a multiple of \(t^{-n}\).  Define the linear operator $L$ by
\[
    Lz
    :=
    z''
    +
    \Big((n-2)+\frac12e^{-2t}\Big)z'
    -
    (n-1)z.
\]
After increasing \(T_0\) if necessary, we have
\begin{equation}\label{eq:Lt-minus-n}
    L(t^{-n})\le -c_0 t^{-n},
    \quad t\ge T_0
\end{equation}
for some \(c_0>0\).
Let \(C_1\) be the constant appearing in \eqref{eq:X-diff-ineq}.
Choose \(K>0\) sufficiently large so that $Kc_0\ge C_1$ and $X(T_0)\le K T_0^{-n}$.
Then, by \eqref{eq:X-diff-ineq} and \eqref{eq:Lt-minus-n}, we have $ X(T_0)-K T_0^{-n}\le0$ and $L\bigl(X-K t^{-n}\bigr)\ge0$ for $t\ge T_0$. Applying the maximum principle yields $X(t)\le K t^{-n}$ for $t\ge T_0$. This is equivalent to
\begin{equation*}
    \|\phi(t,\cdot)-\bar\phi(t)\|_{L^2(\mathbb{S}^{n-1})}
    \le C t^{-n/2}.
\end{equation*}
Lemma \ref{lem:spherical-oscillation} follows immediately from this estimate.
\end{proof}

\begin{lemma}\label{lem:uniform-bounds}
There exist constants \(C>0\) and \(T>0\) such that for all  $t\ge T$ and $\theta\in\mathbb{S}^{n-1}$,
\begin{equation}\label{eq:v-bounded}
    0\le v(t,\theta)\le C
\end{equation}
and
\begin{equation}\label{eq:v-derivatives-bounded}
    |v_t(t,\theta)|
    +
    |v_{tt}(t,\theta)|
    +
    |\nabla_\theta v(t,\theta)|
    \le C.
\end{equation}
\end{lemma}

\begin{proof}
By Lemma~\ref{lem:spherical-average}, we have \eqref{eq:v-bounded}. A direct calculation gives
\begin{equation}\label{eq:v-equation}
\begin{aligned}
    v_{tt}+\Delta_{\mathbb S^{n-1}} v
    &+
    \left((n-2)(1-\frac1t)+\frac12e^{-2t}\right)v_t        \\
    &-
    \left(
        \frac{(n-2)^2}{2t}
        -
        \frac{n(n-2)}{4t^2}
        +
        \frac{n-2}{4t}e^{-2t}
    \right)v
    =
    -\frac1t v^{\frac{n}{n-2}}.
\end{aligned}
\end{equation}
After the translation \(s:=t-\tau\), set
\[
    V_\tau(s,\theta):=v(\tau+s,\theta),
    \quad (s,\theta)\in (-1,1)\times \mathbb S^{n-1}.
\]
Then \(V_\tau\) satisfies, on the fixed cylinder $Q:=(-1,1)\times \mathbb S^{n-1}$,
a uniformly elliptic equation whose lower-order coefficients are bounded in
\(C^1(Q)\). Moreover, the right-hand side is uniformly bounded in \(L^\infty(Q)\). Applying the interior
\(W^{2,p}\)-estimates and the Schauder estimates on $Q_{1/2}:=(-\tfrac12,\tfrac12)\times\mathbb S^{n-1}$,
we obtain that $\|V_\tau\|_{C^{2,\alpha}(Q_{1/2})}\le C$ for some
\(\alpha\in(0,1)\). Returning to \(t=\tau+s\) yields
\[
    \|v\|_{C^{2,\alpha}((\tau-\frac12,\tau+\frac12)\times\mathbb S^{n-1})}
    \le C,
\]
where \(C\) is independent of \(\tau\).
Hence, $|v_t|+|v_{tt}|+|\nabla_\theta v|\le C$. This proves \eqref{eq:v-derivatives-bounded}.
\end{proof}

The following lemma establishes the square integrability of \(v_t\) and \(\nabla_\theta v\) on the tail of the cylinder, as well as the convergence \(v_t(t,\cdot)\to0\) in \(L^2(\mathbb{S}^{n-1})\).

\begin{lemma}\label{lem:energy-estimates}
There exists \(T>0\) such that
\begin{equation}\label{eq:vt-L2-finite}
    \int_T^\infty\int_{\mathbb{S}^{n-1}}v_t^2\,d\theta\,d t<\infty,
\end{equation}
\begin{equation}\label{eq:grad-L2-finite}
    \int_T^\infty\int_{\mathbb{S}^{n-1}}|\nabla_\theta v|^2\,d\theta\,d t<\infty
\end{equation}
and
\begin{equation}\label{eq:vt-L2-limit-zero}
    \lim_{t\to\infty}
    \int_{\mathbb{S}^{n-1}}v_t(t,\theta)^2\,d\theta=0.
\end{equation}
\end{lemma}

\begin{proof}
Throughout the proof, \(T>0\) is chosen sufficiently large and may be
increased finitely many times.  We write
\[
    I(t):=\int_{\mathbb{S}^{n-1}}v_t^2\,d\theta,
    \qquad
    G(t):=\int_{\mathbb{S}^{n-1}}|\nabla_\theta v|^2\,d\theta,
\]
\[
    V(t):=\int_{\mathbb{S}^{n-1}}v^2\,d\theta,
    \qquad
    F(t):=\int_{\mathbb{S}^{n-1}}v^{\frac{n}{n-2}+1}\,d\theta.
\]
We first verify \eqref{eq:vt-L2-finite}. Let
\begin{equation*}
\begin{aligned}
    \mathcal E(t)
    :=
    &\frac12 I(t)
    -
    \frac12 G(t)        \\
    &-
    \frac12
    \left(
        \frac{(n-2)^2}{2t}
        -
        \frac{n(n-2)}{4t^2}
        +
        \frac{n-2}{4t}e^{-2t}
    \right)V(t)
    +
    \frac1{\left(\frac{n}{n-2}+1\right)t}F(t).
\end{aligned}
\end{equation*}
Multiplying \eqref{eq:v-equation} by \(v_t\) and integrating over \(\mathbb{S}^{n-1}\), we get
\begin{equation}\label{eq:energy-identity}
\begin{aligned}
    &\frac{d}{d t}\mathcal E(t)
    +
    \left((n-2)\left(1-\frac1t\right)+\frac12e^{-2t}\right)I(t)        \\
    & \quad =
    -
    \frac12
    \frac{d}{d t}
    \left(
        \frac{(n-2)^2}{2t}
        -
        \frac{n(n-2)}{4t^2}
        +
        \frac{n-2}{4t}e^{-2t}
    \right)V(t)       -
    \frac1{\left(\frac{n}{n-2}+1\right)t^2}F(t).
\end{aligned}
\end{equation}
By Lemma~\ref{lem:uniform-bounds}, the quantities \(I(t)\), \(G(t)\),
\(V(t)\), and \(F(t)\) are uniformly bounded.  Hence
\begin{equation}\label{eq:energy-bounded}
    |\mathcal E(t)|\le C.
\end{equation}
Moreover, the right-hand side of \eqref{eq:energy-identity} is bounded in
absolute value by \(Ct^{-2}\), and hence is integrable on \([T,\infty)\).
After increasing \(T\) if necessary,
\[
    (n-2)\left(1-\frac1t\right)+\frac12e^{-2t}
    \ge \frac{n-2}{2}>0,
    \quad t\ge T.
\]
Integrating \eqref{eq:energy-identity} from \(T\) to \(S\), using
\eqref{eq:energy-bounded} and letting \(S\to\infty\), we get $\int_T^\infty I(t)\,d t<\infty$.
This proves \eqref{eq:vt-L2-finite}.

Now we prove \eqref{eq:grad-L2-finite}.
Let
$
    h(t,\theta):=v(t,\theta)-\bar v(t).
$
By Lemma~\ref{lem:spherical-oscillation},
\begin{equation}\label{eq:h-decay}
    \|h(t,\cdot)\|^2_{L^2(\mathbb{S}^{n-1})}\le Ct^{-2}.
\end{equation}
Multiplying \eqref{eq:v-equation} by \(h\) and integrating over \(\mathbb{S}^{n-1}\) yields
\begin{equation}\label{eq:gradient-identity}
\begin{aligned}
    G(t)
    =
    &\int_{\mathbb{S}^{n-1}}v_{tt}h\,d\theta    +
    \left((n-2)(1-\frac1t)+\frac12e^{-2t}\right)
    \int_{\mathbb{S}^{n-1}}v_th\,d\theta        \\
    &-
    \left(
        \frac{(n-2)^2}{2t}
        -
        \frac{n(n-2)}{4t^2}
        +
        \frac{n-2}{4t}e^{-2t}
    \right)
    \int_{\mathbb{S}^{n-1}}h^2\,d\theta +\frac1t\int_{\mathbb{S}^{n-1}}v^{\frac{n}{n-2}} h\,d\theta.
\end{aligned}
\end{equation}
We estimate the integral in time of each term on the right-hand side.
First,
\[
    \int_T^S\int_{\mathbb{S}^{n-1}}v_{tt}h\,d\theta\,d t
    =
    \left[\int_{\mathbb{S}^{n-1}}v_th\,d\theta\right]_{T}^{S}
    -
    \int_T^S\int_{\mathbb{S}^{n-1}}v_t h_t\,d\theta\,d t.
\]
By Lemmas \ref{lem:uniform-bounds} and \ref{lem:spherical-oscillation}, the boundary term is uniformly bounded. Moreover,
\[
0\le \int_{\mathbb{S}^{n-1}}v_th_t \,d\theta
=
\int_{\mathbb{S}^{n-1}}v_t^2-|\mathbb{S}^{n-1}|\bar v_t^2
\le
\int_{\mathbb{S}^{n-1}}v_t^2 \,d\theta.
\]
By \eqref{eq:vt-L2-finite}, it follows that
\begin{equation*}
    \left|
    \int_T^S\int_{\mathbb{S}^{n-1}}v_{tt}h\,d\theta\,d t
    \right|
    \le C,
\end{equation*}
where \(C\) is independent of \(S\). Applying Cauchy's inequality yields
\[
\begin{aligned}
    &\int_T^\infty
    \left|
        (n-2)\left(1-\frac1t\right)+\frac12e^{-2t}
    \right|
   \times
    \left|
        \int_{\mathbb{S}^{n-1}}v_th\,d\theta
    \right|\,d t        \\
    &\le
    C(n,T)
    \left(
        \int_T^\infty I(t)\,d t
    \right)^{1/2}
    \left(
        \int_T^\infty\|h(t,\cdot)\|_{L^2(\mathbb{S}^{n-1})}^2\,d t
    \right)^{1/2}
    <\infty.
\end{aligned}
\]
Moreover, \eqref{eq:h-decay} gives
\[
    \int_T^\infty
    \left|
        \frac{(n-2)^2}{2t}
        -
        \frac{n(n-2)}{4t^2}
        +
        \frac{n-2}{4t}e^{-2t}
    \right|
    \int_{\mathbb{S}^{n-1}}h^2\,d\theta\,d t
    \le
    C\int_T^\infty t^{-1}t^{-2}\,d t
    <\infty.
\]
By Lemma \ref{lem:uniform-bounds}, \(v\) is uniformly bounded, which implies
\[
     \int_T^\infty \left|
        \frac1t\int_{\mathbb{S}^{n-1}}v^{\frac{n}{n-2}}h\,d\theta
    \right|
    \le
     \int_T^\infty \frac C t\|h(t,\cdot)\|_{L^2(\mathbb{S}^{n-1})}
    \le
    C  \int_T^\infty t^{-2}  <\infty.
\]
Integrating \eqref{eq:gradient-identity} from \(T\) to \(S\), using the
bounds above and letting \(S\to\infty\), we obtain \eqref{eq:grad-L2-finite}.

By Lemma \ref{lem:uniform-bounds}, we have
\[
    |I'(t)|
    =
    2\left|
        \int_{\mathbb{S}^{n-1}}v_t v_{tt}\,d\theta
    \right|
    \le C.
\]
On the other hand, \eqref{eq:vt-L2-finite} shows that $\int_T^\infty I(t)\,d t<\infty$. Thus \(I(t)\to0\), which proves \eqref{eq:vt-L2-limit-zero}. The proof of Lemma \ref{lem:energy-estimates} is now complete.
\end{proof}

The following lemma establishes a dichotomy for the asymptotic behavior of solutions.

\begin{lemma}\label{lem:cylindrical-limit}
Let $q=\frac{n}{n-2}$ and $u\in C^2(B_1\setminus\{0\})$ be a nonnegative solution of
\eqref{eq:main}. Then
\begin{equation*}
\lim_{x\to0}
    |x|^{n-2}(-\ln |x|)^{\frac{n-2}{2}}u(x)
    =
    0
    ~~\text{or}~~
    \left(\frac{n-2}{\sqrt2}\right)^{n-2}.
\end{equation*}
\end{lemma}

\begin{proof}

Averaging \eqref{eq:v-equation} over \(\mathbb{S}^{n-1}\), we obtain
\begin{equation}\label{eq:vbar-equation}
    (\bar v)''
    +
    \left((n-2)(1-\frac1t)+\frac12e^{-2t}\right)(\bar v)'
    =
    \frac1t f(\bar v)+R(t),
\end{equation}
where
\begin{equation}\label{eq:R-def}
    f(s):=\frac{(n-2)^2}{2}s-s^{\frac{n}{n-2}} ~~ \text{and} ~~  R(t)
    :=
    \frac1t\bigl((\bar v)^{\frac{n}{n-2}}-\overline{v^{\frac{n}{n-2}}}\bigr)
    -
    \frac{n(n-2)}{4t^2}\bar v
    +
    \frac{n-2}{4t}e^{-2t}\bar v.
\end{equation}
We claim that
\begin{equation}\label{eq:R-L1}
    R\in L^1([T,\infty))\qquad \text{for \(T\) large enough.}
\end{equation}
By Lemma \ref{lem:uniform-bounds}, \(v\) is
uniformly bounded.  Hence the map \(s\mapsto s^{\frac{n}{n-2}}\) is Lipschitz on the range of
\(v\), and so
\[
    \left|
        \overline{v^{\frac{n}{n-2}}}-(\bar v)^{\frac{n}{n-2}}
    \right|
    \le
    C(n)
    \frac1{|\mathbb{S}^{n-1}|}
    \int_{\mathbb{S}^{n-1}}|v-\bar v|\,d\theta
    \le
    C(n)\|v(t,\cdot)-\bar v(t)\|_{L^2(\mathbb{S}^{n-1})}.
\]
Combining this with Lemma \ref{lem:spherical-oscillation} gives
\[
    \left|
        \frac1t\bigl((\bar v)^{\frac{n}{n-2}}-\overline{v^{\frac{n}{n-2}}}\bigr)
    \right|
    \le Ct^{-2}.
\]
The remaining two terms in \eqref{eq:R-def} are also integrable.
This proves \eqref{eq:R-L1}.

It follows from Lemma~\ref{lem:energy-estimates} that
\begin{equation}\label{eq:vbarprime-zero}
    |(\bar v)'(t)|^2
    =
    \left|
        \frac1{|\mathbb{S}^{n-1}|}\int_{\mathbb{S}^{n-1}}v_t(t,\theta)\,d\theta
    \right|^2
    \le
    \frac1{|\mathbb{S}^{n-1}|}
    \int_{\mathbb{S}^{n-1}}v_t(t,\theta)^2\,d\theta
    \to0\quad \text{as }t\to +\infty.
\end{equation}
Moreover, we have $(\bar v)'\in L^2([T,\infty))$. This implies that
\begin{equation}\label{eq:drift-error-yprime-L1}
    \left(-\frac{n-2}{t}+\frac12e^{-2t}\right)(\bar v)'
    \in L^1([T,\infty)).
\end{equation}
By \eqref{eq:lemma1-limsup}, we have
\begin{equation*}
    \limsup_{t\to\infty}\bar v(t)
    \le
    \left(\frac{n-2}{\sqrt2}\right)^{n-2}.
\end{equation*}
We now prove that \(\bar v(t)\) has a limit.  Suppose, by contradiction, that
\(\bar v\) does not converge. Choose $\lambda,\,\mu \geq 0$ such that
\[
    0 \le \liminf_{t\to\infty}\bar v(t)<\lambda<\mu<\limsup_{t\to\infty}\bar v(t) \le \left(\frac{n-2}{\sqrt2}\right)^{n-2}.
\]
We may choose intervals
$
    [s_k,t_k]\subset [T,\infty)  ~ \text{with} ~ s_k, t_k\to\infty,
$
such that
$
    \bar v(s_k)=\mu, \bar v(t_k)=\lambda
$
and
$
    \lambda\le \bar v(t)\le\mu
    \text{ for }s_k\le t\le t_k.
$
Integrating \eqref{eq:vbar-equation} from \(s_k\) to \(t_k\) gives
\begin{align}
    &(\bar v)'(t_k)-(\bar v)'(s_k)
    +(n-2)\bigl(\bar v(t_k)-\bar v(s_k)\bigr)
    +\int_{s_k}^{t_k}
    \left(-\frac{n-2}{t}+\frac12e^{-2t}\right)(\bar v)'(t)\,d t        \notag\\
    &
    =
    \int_{s_k}^{t_k}\frac{f(\bar v(t))}{t}\,d t
    +
    \int_{s_k}^{t_k}R(t)\,d t.
    \label{eq:crossing-integral}
\end{align}
Using \eqref{eq:vbarprime-zero}, \eqref{eq:drift-error-yprime-L1} and \eqref{eq:R-L1}, we obtain that as \(k\to\infty\),
\[
    (\bar v)'(t_k)-(\bar v)'(s_k)\to0,\quad
    \int_{s_k}^{t_k}
    \left(-\frac{n-2}{t}+\frac12e^{-2t}\right)(\bar v)'(t)\,d t
    \to0,\quad
    \int_{s_k}^{t_k}R(t)\,d t\to0.
\]
Passing to the limit in \eqref{eq:crossing-integral}, we have
\begin{equation}\label{eq:crossing-limit}
    0>(n-2)(\lambda-\mu)
    =
    \lim_{k\to\infty}
    \int_{s_k}^{t_k}\frac{f(\bar v(t))}{t}\,d t.
\end{equation}
However, $f(s)>0$ for \(0<s<\bigl((n-2)/\sqrt2\bigr)^{n-2}\).
Hence the right-hand side of
\eqref{eq:crossing-limit} is nonnegative.
This contradiction shows that \(\bar v(t)\) must have a limit.  We denote it by
\begin{equation}\label{eq:v-uniform-limit}
     \ell:=\lim\limits_{t\to\infty}\bar v(t).
\end{equation}
Now we show that $v(t,\theta)$ converges uniformly to $\ell$ as $t \to \infty$. We claim that
\begin{equation}\label{eq:L2-to-Linfty}
    \|v(t,\cdot)-\bar v(t)\|_{L^\infty(\mathbb{S}^{n-1})}\to0.
\end{equation}
If not, there exist \(\varepsilon_0>0\), a sequence \(t_j\to\infty\), and points \(\theta_j\in\mathbb S^{n-1}\) such that $|v(t_j,\theta_j)-\bar v(t_j)|\ge \varepsilon_0$.
Set $f_j(\theta):=v(t_j,\theta)-\bar v(t_j)$. By Lemma~\ref{lem:uniform-bounds}, we know $\|\nabla_\theta(v(t,\cdot)-\bar v(t))\|_{L^\infty(\mathbb{S}^{n-1})}\le C$. Hence, the inequality
\[
    |f_j(\theta)-f_j(\theta_j)|
    \le C\,d_{\mathbb S^{n-1}}(\theta,\theta_j)
\]
holds for  \(\theta\in\mathbb S^{n-1}\). Here \(d_{\mathbb S^{n-1}}\) denotes the geodesic distance on the sphere.
Let
\[
\rho:=\min\left\{1,\frac{\varepsilon_0}{2(C+1)}\right\}.
\]
Then, whenever \(d_{\mathbb S^{n-1}}(\theta,\theta_j)<\rho\), we have the lower bound
\[
    |f_j(\theta)|
    \ge |f_j(\theta_j)|-|f_j(\theta)-f_j(\theta_j)|
    \ge \varepsilon_0-C\rho
    \ge \frac{\varepsilon_0}{2}.
\]
Since each geodesic ball $B_{\mathbb S^{n-1}}(\theta_j,\rho)$ has measure bounded below by a positive constant \(m_\rho>0\),
we obtain
\[
\begin{aligned}
\|v(t_j,\cdot)-\bar v(t_j)\|_{L^2(\mathbb S^{n-1})}^2
&\ge
\int_{B_{\mathbb S^{n-1}}(\theta_j,\rho)}
|f_j(\theta)|^2\,d\theta        \ge
\frac{\varepsilon_0^2}{4}m_\rho>0.
\end{aligned}
\]
This contradicts Lemma~\ref{lem:spherical-oscillation}.  Hence \eqref{eq:L2-to-Linfty}
holds. Consequently,
\[
    \sup_{\theta\in\mathbb{S}^{n-1}}|v(t,\theta)-\ell|
    \le
    \|v(t,\cdot)-\bar v(t)\|_{L^\infty(\mathbb{S}^{n-1})}
    +
    |\bar v(t)-\ell|
    \to 0.
\]
Thus, $v(t,\theta)$ converges uniformly to $\ell$ as $t \to \infty$.

Finally, we identify the possible values of \(\ell\).  Integrating
\eqref{eq:vbar-equation} from \(T\) to \(S\) gives
\begin{align*}
    \int_T^S\frac{f(\bar v(t))}{t}\,d t
    &=
    (\bar v)'(S)-(\bar v)'(T)
    +(n-2)\bigl(\bar v(S)-\bar v(T)\bigr)        \notag\\
    &\quad
    +\int_T^S
    \left(-\frac{n-2}{t}+\frac12e^{-2t}\right)(\bar v)'(t)\,d t
    -
    \int_T^S R(t)\,d t.
\end{align*}
By \eqref{eq:v-uniform-limit},
\eqref{eq:R-L1}, \eqref{eq:vbarprime-zero} and \eqref{eq:drift-error-yprime-L1}, the right-hand side has a finite limit as \(S\to\infty\).
Therefore, the improper integral
\[
    \int_T^\infty\frac{f(\bar v(t))}{t}\,d t
\]
is finite. This leads to $ \ell=0 $ or $\ell=\left(\frac{n-2}{\sqrt2}\right)^{n-2}$. This completes the proof of Lemma \ref{lem:cylindrical-limit}.
\end{proof}

Finally, we show that the singularity is removable when $\lim_{x\to0} |x|^{n-2}(-\ln |x|)^{\frac{n-2}{2}}u(x)=0$.

\begin{lemma}\label{lem:zero-limit-removable}
Let $q=\frac{n}{n-2}$ and $u\in C^2(B_1\setminus\{0\})$ be a nonnegative solution of
\eqref{eq:main}. If
\begin{equation}\label{eq:lzero-asymptotic}
    \lim_{x\to0}
    |x|^{n-2}(-\ln |x|)^{\frac{n-2}{2}}u(x)=0,
\end{equation}
then $u$ can be extended as a $C^2(B_1)$ solution of \eqref{eq:main} in $B_1$.
\end{lemma}

\begin{proof}
Set $\tau(x):=-\ln |x|^2$.
After decreasing the radius, we may assume \(\tau>1\) in the punctured ball. The assumption is equivalent to
\begin{equation}\label{eq:uqminus1-small}
    u(x)^{2/(n-2)}
    =
    o\left(|x|^{-2}\tau(x)^{-1}\right)
    \quad\text{as }x\to0.
\end{equation}
Choose $0<s<\min\left\{1,\frac{n-2}{2}\right\}$ and define
\begin{equation}\label{eq:weighted-u-def-removable}
    \widetilde u(x):=\tau(x)^{-s}u(x).
\end{equation}
A direct computation gives
\begin{equation}\label{eq:weighted-u-equation}
    L_s \widetilde u=Q_s(x)\widetilde u
\end{equation}
where
\begin{equation*}
    L_s
    :=
    \Delta
    -
    \left(
        \frac{4s}{|x|^2\tau(x)}x + \frac12x
    \right)\cdot\nabla
\end{equation*}
and
\begin{equation*}
    Q_s(x)
    :=
    \frac{2s(n-2)}{|x|^2\tau(x)}
    +
    \frac{4s(1-s)}{|x|^2\tau^2(x)}
    -
    \frac{s}{\tau(x)}
    +
    \frac{n-2}{2}
    -
    u(x)^{2/(n-2)}.
\end{equation*}
By \eqref{eq:uqminus1-small}, for some small \(r_0>0\) we have
\[
    u(x)^{2/(n-2)}
    \le
    \frac{s(n-2)}{2|x|^2\tau(x)},
    \quad 0<|x|<r_0.
\]
After decreasing \(r_0>0\) if necessary, $Q_s(x)\ge \frac{s(n-2)}{4|x|^2\tau(x)}>0$ for $0<|x|<r_0$. Thus, \eqref{eq:weighted-u-equation} gives
\begin{equation}\label{eq:Ls-weighted-u-positive}
    L_s \widetilde u\ge0
    \quad\text{in }B_{r_0}\setminus\{0\}.
\end{equation}
For \(0<r<r_0\), define
\begin{equation*}
    \Psi(r)
    :=
    \int_r^{r_0}
        \rho^{1-n}\tau(\rho)^{-2s}e^{\rho^2/4}\,d\rho.
\end{equation*}
Then a straightforward computation yields
\begin{equation}\label{eq:Ls-Psi-zero}
    L_s\Psi=0
    \quad\text{in }B_{r_0}\setminus\{0\}.
\end{equation}
Moreover,
\begin{equation}\label{eq:Psi-asymptotic}
    \Psi(r)
    \sim
    \frac{1}{n-2} r^{-(n-2)}\tau(r)^{-2s}
    \quad\text{as } r \to 0.
\end{equation}
By \eqref{eq:lzero-asymptotic} and \eqref{eq:weighted-u-def-removable}, we have
\[
    \widetilde u(x)
    =
    o\left(|x|^{-(n-2)}\tau(x)^{-(n-2)/2-s}\right)
    \quad\text{as }x\to0.
\]
Combining this with \eqref{eq:Psi-asymptotic}, we obtain
\begin{equation}\label{eq:weighted-u-over-Psi-zero}
    \frac{\widetilde u(x)}{\Psi(|x|)}
    =
    o\left(\tau(x)^{s-(n-2)/2}\right)
    \to 0
    \quad\text{as }x \to 0.
\end{equation}
Let $M:=\max\limits_{|x|=r_0}\widetilde u(x)$. Choose a sequence \(\varepsilon_j \to 0\).  By
\eqref{eq:weighted-u-over-Psi-zero}, for each \(j\) we can choose
$
    0<r_j<\min\left\{1/j,r_0\right\}
$
such that $ \widetilde u(x)\le \varepsilon_j\Psi(|x|) $ for $|x|=r_j$. Set
\[
    A_j:=\{x:r_j<|x|<r_0\}
    \quad \text{and} \quad
    H_j(x):=\widetilde u(x)-\varepsilon_j\Psi(|x|)-M.
\]
By \eqref{eq:Ls-weighted-u-positive} and \eqref{eq:Ls-Psi-zero}, we have
\[
    L_sH_j\ge0
    \quad\text{in } A_j.
\]
The definition of \(M\) and the choice of \(r_j\) give $ H_j\le0$ on $\partial A_j$. Then, the maximum principle yields
\begin{equation}\label{eq:76}
      \widetilde u(x)\le M+\varepsilon_j\Psi(|x|)
    \quad\text{for }r_j<|x|<r_0.
\end{equation}
Letting \(j\to\infty\)
in \eqref{eq:76} gives \(\widetilde u(x)\le M\).
Hence,
\begin{equation}\label{eq:u-log-growth}
    u(x)\le C\tau(x)^s
    \quad\text{near the origin}.
\end{equation}

Let \(x_0\ne0\) be sufficiently close to the origin and set $r:=|x_0|$. Define
\[
    U(y):=u\left(x_0+\frac r4 y\right),
    \quad |y|<1.
\]
Then \(U\) satisfies
\[
    -\Delta_y U
    +
    \frac r8
    \left(x_0+\frac r4 y\right)\cdot\nabla_y U
    +
    \frac{(n-2) r^2}{32}U
    =
    \frac{r^2}{16}U^{\frac{n}{n-2}}
    \quad\text{in }B_1.
\]
By the interior gradient estimates, we have $|\nabla_y U(0)| \le C(-\ln r)^s$. Returning to the original variables, we obtain
\begin{equation}\label{eq:xgrad-log-bound}
    |x||\nabla u(x)|
    \le C(-\ln |x|)^s \quad \text {near the origin}.
\end{equation}
Set $F(x):= \frac12x\cdot\nabla u+\frac{n-2}{2}u-u^{\frac{n}{n-2}}$. Then \eqref{eq:u-log-growth} and \eqref{eq:xgrad-log-bound} imply that $F\in L^p(B_{r_0})$ for every  $1\le p<\infty$. By Lemma \ref{lem:local-integrability-distribution} and the standard elliptic estimates, we obtain $u\in C_{\mathrm{loc}}^{2}(B_{r_0})$. This completes the proof of Lemma \ref{lem:zero-limit-removable}.
\end{proof}

\begin{proof}[Proof of Theorem \ref{thm:main2}]
It follows from Lemma~\ref{lem:cylindrical-limit} and Lemma~\ref{lem:zero-limit-removable}.
\end{proof}

\section{Proof of Theorem~\ref{thm:supercritical-alternative}}\label{sec:supercritical}
In this section, we apply Simon's convergence theorem to prove Theorem \ref{thm:supercritical-alternative}. Specifically, we shall use a slightly reformulated
scalar form of Simon \cite[Part II, Theorem 1.2]{simon1985isolated}, adapted to the present setting (see also Simon \cite[Theorem 1]{simon1983asymptotics}). Although the theorem is stated under a global uniform analyticity condition, its proof only requires this condition in a neighborhood of the limiting stationary point and of the trajectory. Therefore, it suffices to assume local uniform analyticity.

Since \(q>\frac{n+2}{n-2}\) in the present case, the coefficients in \eqref{eq:cylindrical-coefficient} satisfy \(k_0>0\) and \(k_1<0\). If
\begin{equation}\label{eq:supercritical-zero-limit}
    \lim_{|x|\to0}|x|^{\frac{2}{q-1}}u(x)=0,
\end{equation}
then Lemma \ref{lem:zero-removable} shows that $u$ can be extended as a $C^2(B_1)$ solution of \eqref{eq:main} in $B_1$. Thus we only need to consider the case
\begin{equation}\label{eq:supercritical-nonzero-branch}
   \limsup_{|x|\to 0} |x|^{\frac{2}{q-1}}u(x)>0.
\end{equation}

\begin{lemma}\label{lem:supercritical-singular-convergence}
Let $q>\frac{n+2}{n-2}$ and $u\in C^2(B_1\setminus\{0\})$ be a nonnegative solution of
\eqref{eq:main}. Assume that \eqref{eq:supercritical-scale-bound} and \eqref{eq:supercritical-nonzero-branch} hold. Then there exists a
positive function \(\omega\in C^\infty(\mathbb S^{n-1})\) satisfying
\eqref{eq:supercritical-spherical-profile} such that
\[
  \lim_{r\to 0}  r^{\frac{2}{q-1}}u(r,\cdot) = \omega(\cdot)
    \quad\text{in }C^k(\mathbb S^{n-1})
\]
for every integer \(k\ge0\).
\end{lemma}

\begin{proof}
Let $w$ be defined as in \eqref{eq:cylindrical-rescaling}. By \eqref{eq:supercritical-scale-bound}, \(w\) satisfies the hypothesis \eqref{eq:bdd} of Lemma~\ref{lem:cylindrical-energy-limit}. Therefore, the
limit $\bar E$ exists and is finite.
Let \(s_j\to-\infty\) be arbitrary.
By Lemma~\ref{lem:constant-energy-stationarity}, after passing
to a subsequence, there exists a nonnegative function
\(W\in C^2(\mathbb S^{n-1})\) satisfying
\eqref{eq:limit} such that
\begin{equation*}
    w(t+s_j,\theta)\to W(\theta)
    \quad\text{in }C^2_{\mathrm{loc}}
    (\mathbb R\times\mathbb S^{n-1}),
\end{equation*}
and for every fixed
\(t\in\mathbb R\),
\[
    E(t+s_j;w) \to E(t;W) \equiv \bar E.
\]
Plugging \eqref{eq:cylindrical-coefficient} into \eqref{eq:limit}, we get that $W$ satisfies \eqref{eq:supercritical-spherical-profile}.  Multiplying \eqref{eq:supercritical-spherical-profile} by \(W\)
and integrating over \(\mathbb S^{n-1}\), we obtain
\begin{equation*}
    \bar E=E(t;W)
    =
    \left(\frac12-\frac{1}{q+1}\right)
    \int_{\mathbb S^{n-1}}W^{q+1}\,d\theta\ge0.
\end{equation*}
Hence, \(\bar E=0\) if and only if \(W\equiv0\). If \(\bar E=0\), then \eqref{eq:supercritical-zero-limit} holds, contradicting \eqref{eq:supercritical-nonzero-branch}. Consequently, \(\bar E>0\). We claim that there exist
$\delta>0$ and $T_1<0$ such that
\begin{equation}\label{eq:energy-method-tail-positive}
    w(t,\theta)\ge\delta
    \qquad
    \text{for all }t\le T_1  \text{ and } \theta\in\mathbb S^{n-1}.
\end{equation}
Suppose otherwise. Then there exist $t_j\to-\infty$ and
$\theta_j\in\mathbb S^{n-1}$ such that
$
    w(t_j,\theta_j) \to 0.
$
After passing to a subsequence, $\theta_j\to\theta_0$.  Applying
Lemma~\ref{lem:constant-energy-stationarity} to the sequence
$\{t_j\}$, and passing to a further subsequence, we obtain
\[
    w(t+t_j,\theta)\to W(\theta)
    \quad\text{in }C^2_{\mathrm{loc}}
    (\mathbb R\times\mathbb S^{n-1}),
\]
where $W\ge 0$ satisfies \eqref{eq:limit} and $W(\theta_0)=0$.  By the Harnack inequality, $W\equiv0$.  Hence $\bar E=0$, a contradiction. This proves
\eqref{eq:energy-method-tail-positive}.

By \eqref{eq:supercritical-scale-bound}, after decreasing $T_1$ if
necessary, there exists $M_*>0$ such that
\begin{equation*}
    \delta\le w(t,\theta)\le M_*
    \qquad
    \text{for all }t\le T_1 \text{ and }  \theta\in\mathbb S^{n-1}.
\end{equation*}
Define
\[
    V(s,\theta):=w(T_1-s,\theta),
    \quad s\ge0.
\]
Then $s\to+\infty$ corresponds to $t\to-\infty$, and
\eqref{eq:general-cylinder-equation} becomes
\begin{equation}\label{eq:energy-method-reversed-equation}
\begin{aligned}
    &V_{ss}
    -\left(n-2-\frac{4}{q-1}\right)V_s
    +\Delta_{\mathbb S^{n-1}} V
    -\frac{2}{q-1}
       \left(n-2-\frac{2}{q-1}\right)V \\
    &\qquad
    +V^q
    +ce^{-2s}V_s=0
    \quad \text{in }(0,\infty)\times\mathbb S^{n-1},
\end{aligned}
\end{equation}
where $c:=\frac12e^{2T_1}>0$. Moreover,
\begin{equation}\label{eq:energy-method-V-two-sided}
    \delta\le V(s,\theta)\le M_*
   \qquad
    \text{for all }s\ge0 \text{ and }   \theta\in\mathbb S^{n-1}.
\end{equation}
Set
\[
    F_0(z,p,\tau)
    :=
    \frac12\tau^2
    +\frac12|p|^2
    +\frac{1}{q-1}
       \left(n-2-\frac{2}{q-1}\right)z^2
    -\frac{1}{q+1}z^{q+1}
\]
and
\[
    \rho(s):=\exp\!\left(-\frac c2e^{-2s}\right).
\]
For $0<A<B<\infty$, define
\[
    \mathcal G_{A,B}(\psi)
    :=
    \int_A^B\int_{\mathbb S^{n-1}}
    e^{-\left(n-2-\frac{4}{q-1}\right)s}
    \rho(s)
    F_0(\psi,\nabla_\theta\psi,\psi_s)
    \,d\theta\,ds.
\]
For any \(\eta\in C_c^\infty((A,B)\times\mathbb S^{n-1})\), integrating by parts with respect to \(s\) yields
\[
\begin{aligned}
    \left.
    \frac{d}{d\varepsilon}
    \mathcal G_{A,B}(V+\varepsilon\eta)
    \right|_{\varepsilon=0}
    &= \int_A^B\int_{\mathbb S^{n-1}}
    e^{-\left(n-2-\frac{4}{q-1}\right)s}\rho(s)        \\
    &\qquad\times
    \left[
        -V_{ss}
        +
        \left(
            n-2-\frac{4}{q-1}
            -
            \frac{\rho'}{\rho}
        \right)V_s
        -
        \Delta_{\mathbb S^{n-1}} V \right.   \\
       &\qquad\qquad \left.+
        \frac{2}{q-1}
        \left(
            n-2-\frac{2}{q-1}
        \right)V
        -
        V^q
    \right]\eta
    \,d\theta\,d s.
\end{aligned}
\]
Hence, the Euler-Lagrange equation of this functional is
exactly \eqref{eq:energy-method-reversed-equation}. Write
\[
    \rho(s)F_0=F_0+E
    ~~ \text{ with }
    E(s,z,p,\tau):=(\rho(s)-1)F_0(z,p,\tau).
\]
Then, on every compact subset of the relevant range, we have
\[
    |E|+|DE|+|D^2E|
    \le Ce^{-2s}.
\]
The functions $z^q$ and $z^{q+1}$ are real analytic in a neighborhood of the closure of the range of $V$ given in \eqref{eq:energy-method-V-two-sided}. Furthermore,
\[
    \tau\partial_\tau F_0=\tau^2>0 ~ (\tau\ne0),
    \quad
    \partial_{\tau\tau}^2F_0=1
    \quad \text{and} \quad
    \partial_{p_ip_j}^2F_0=g^{ij}.
\]
Applying standard elliptic estimates to \eqref{eq:energy-method-reversed-equation} yields
\[
    \sup_{s\ge1}
    \|V\|_{C^k([s,s+1]\times\mathbb S^{n-1})}<\infty
    \quad \text{for every }k\ge0.
\]
Thus, by Simon's convergence theorem, there exists a function
$\omega\in C^2(\mathbb S^{n-1})$ such that
\[
    V(s,\cdot)\to \omega(\cdot)
    \quad \text{in }C^2(\mathbb S^{n-1})
\]
and
\[
    V_s(s,\cdot)\to0
    \quad \text{in }C^1(\mathbb S^{n-1})
\]
as $s\to +\infty$. Moreover, the limit $\omega$ is positive and satisfies
\eqref{eq:supercritical-spherical-profile}. Furthermore, elliptic regularity implies that
$\omega\in C^\infty(\mathbb S^{n-1})$.

Finally, the higher-order estimates imply that
$\{V(s,\cdot): s\ge1\}$ is precompact in every $C^k$ topology. Since its
$C^2$ limit is uniquely $\omega$, the convergence actually holds in $C^k(\mathbb S^{n-1})$ for all $k\ge0$. Reverting to the original variable $r=e^{T_1-s}$ completes the proof.
\end{proof}

\begin{proof}[Proof of Theorem \ref{thm:supercritical-alternative}]
It follows from Lemmas \ref{lem:zero-removable}  and \ref{lem:supercritical-singular-convergence}.
\end{proof}

\medskip

\noindent{\bf Acknowledgements} Part of this work was completed while H. Yang was visiting The Chinese University of Hong Kong. He would like to thank Professor Juncheng Wei for his support and encouragement, as well as the Department of Mathematics and the Institute of Mathematical Sciences for their hospitality. The research of M. Xu is partially supported by the National Key R\&D Program of China 2025YFA1017600 and NSFC 12526202. The research of H. Yang is supported by NSFC 12301140 and the Shanghai Frontier Science Center of Modern Analysis.

\medskip

\noindent{\bf Data availability} No data were used in this study.

\medskip

\noindent{\bf Conflict of interest} There is no conflict of interest.



\bibliographystyle{amsplain}
\bibliography{refs}

\noindent M. Xu

\noindent  School of Mathematical Sciences, Fudan University\\
Shanghai 200433, China \\[1mm]
Email:  \textsf{meiqing\_xu@fudan.edu.cn}

\bigskip

\noindent   H. Yang

\noindent  School of Mathematical Sciences, Shanghai Jiao Tong University\\
Shanghai 200240, China \\[1mm]
Email:  \textsf{hui-yang@sjtu.edu.cn}

\end{document}